\documentclass[10pt,letterpaper,twoside]{article}

\usepackage[letterpaper,left=1in,right=1in,top=1.5in,bottom=1.5in]{geometry}

\title{Filtrations and cohomology II: the Gauss--Manin connection}\author{Benjamin Antieau}\date{\today}

\usepackage[pdfstartview=FitH,
            pdfauthor={},
            pdftitle={},
            colorlinks,
            linkcolor=reference,
            citecolor=citation,
            urlcolor=e-mail,
            backref]{hyperref}

\usepackage{amsmath}
\usepackage{amscd}
\usepackage{amsbsy}
\usepackage{amssymb}
\usepackage{verbatim}
\usepackage{eufrak}
\usepackage{eucal}
\usepackage{microtype}
\usepackage{hyperref}
\usepackage{mathrsfs}
\usepackage{amsthm}
\usepackage{stmaryrd}
\usepackage{wasysym}
\usepackage{xy}
\input xy
\xyoption{all}
\usepackage{tikz}
\usetikzlibrary{matrix,arrows}
\usepackage{tikz-cd}
\usepackage{dsfont}
\usepackage[T1]{fontenc}
\usepackage{spectralsequences}

\usepackage{fancyhdr}
\usepackage{makeidx}
\makeindex
\newcommand{\defidx}[1]{#1\index{#1|textbf}}
\newcommand{\longdefidx}[2]{#1\index{#2|textbf}}

\newcommand{\stackspace}{2.5}
\newcommand{\stack}[2][1cm]{\;\tikz[baseline, yshift=.65ex]%
    {\foreach \k [evaluate=\k as \r using (.5*#2+.5-\k)*\stackspace] in {1,...,#2}{%
    \ifodd\k{\draw[->](0,\r pt)--(#1,\r pt);}%
    \else{\draw[<-](0,\r pt)--(#1,\r pt);}\fi
    }}\;}

\usepackage[bbgreekl]{mathbbol}
\usepackage{amsfonts}
\DeclareSymbolFontAlphabet{\mathbb}{AMSb} 
\DeclareSymbolFontAlphabet{\mathbbl}{bbold}
\newcommand{\Prism}{{\mathbbl{\Delta}}}

\newcommand{\Inf}{{\mathbbl{\Pi}}}
\newcommand{\Infhat}{\widehat{\Inf}}

\usepackage{color}
\definecolor{todo}{rgb}{1,0,0}
\definecolor{conditional}{rgb}{0,1,0}
\definecolor{e-mail}{rgb}{0,.40,.80}
\definecolor{reference}{rgb}{.20,.60,.22}
\definecolor{mrnumber}{rgb}{.80,.40,0}
\definecolor{citation}{rgb}{0,.40,.80}

\renewcommand{\bf}{\bfseries}

\let\oldmarginpar\marginpar
\renewcommand\marginpar[1]{\-\oldmarginpar[\raggedleft\footnotesize #1]%
{\raggedright\footnotesize #1}}

\newcommand{\Cscr}{\mathcal{C}}
\newcommand{\Dscr}{\mathcal{D}}

\newcommand{\Oscr}{\mathcal{O}}

\renewcommand{\d}{\mathrm{d}}

\newcommand{\E}{\mathrm{E}}
\newcommand{\F}{\mathrm{F}}

\renewcommand{\H}{\mathrm{H}}
\newcommand{\h}{\mathrm{h}}

\renewcommand{\L}{\mathrm{L}}

\newcommand{\R}{\mathrm{R}}
\renewcommand{\r}{\mathrm{r}}

\newcommand{\T}{\mathrm{T}}
\renewcommand{\t}{\mathrm{t}}

\newcommand{\bA}{\mathbf{A}}

\newcommand{\bC}{\mathbf{C}}
\newcommand{\bD}{\mathbf{D}}
\newcommand{\bE}{\mathbf{E}}
\newcommand{\bF}{\mathbf{F}}
\newcommand{\bG}{\mathbf{G}}

\newcommand{\bN}{\mathbf{N}}

\newcommand{\bQ}{\mathbf{Q}}

\newcommand{\bS}{\mathbf{S}}
\newcommand{\bT}{\mathbf{T}}

\newcommand{\bZ}{\mathbf{Z}}

\newcommand{\op}{\mathrm{op}}
\newcommand{\cofib}{\mathrm{cofib}}
\newcommand{\fib}{\mathrm{fib}}

\newcommand{\cn}{\mathrm{cn}}

\newcommand{\Mod}{\mathrm{Mod}}

\newcommand{\FMod}{\mathrm{FMod}}
\newcommand{\BFMod}{\mathrm{BFMod}}
\newcommand{\BF}{\mathrm{BF}}

\newcommand{\Perf}{\mathrm{Perf}}

\newcommand{\cMod}{\mathrm{cMod}}

\newcommand{\DAlg}{\mathrm{DAlg}}

\newcommand{\FDAlg}{\mathrm{FDAlg}}

\newcommand{\Gr}{\mathrm{Gr}}
\newcommand{\gr}{\mathrm{gr}}

\newcommand{\ins}{\mathrm{ins}}
\newcommand{\fil}{\mathrm{fil}}

\newcommand{\LSym}{\mathrm{LSym}}

\newcommand{\Pairs}{\mathrm{Pairs}}

\newcommand{\perf}{\mathrm{perf}}

\renewcommand{\geq}{\geqslant}
\renewcommand{\leq}{\leqslant}

\DeclareMathOperator{\Tor}{Tor}
\DeclareMathOperator{\Ext}{Ext}

\newcommand{\HH}{\mathrm{HH}}
\newcommand{\HP}{\mathrm{HP}}

\renewcommand{\inf}{\mathrm{inf}}
\newcommand{\crys}{\mathrm{crys}}

\newcommand{\dt}{\mathrm{d}t}

\newcommand{\dR}{\mathrm{dR}}
\newcommand{\dRhat}{\widehat{\dR}}

\newcommand{\GM}{\mathrm{GM}}

\newcommand{\Map}{\mathrm{Map}}

\newcommand{\Fun}{\mathrm{Fun}}

\newcommand{\Gm}{\bG_{m}}

\DeclareMathOperator*{\colim}{colim}

\DeclareMathOperator*{\Tot}{Tot}

\DeclareMathOperator{\Spec}{Spec}

\newcommand{\we}{\simeq}
\newcommand{\iso}{\cong}

\theoremstyle{plain}
\newtheorem{theorem}{Theorem}[section]
\newtheorem*{theorem*}{Theorem}
\newtheorem{lemma}[theorem]{Lemma}

\newtheorem{proposition}[theorem]{Proposition}

\newtheorem{corollary}[theorem]{Corollary}
\newtheorem*{corollary*}{Corollary}

\theoremstyle{plain}

\theoremstyle{definition}

\newtheoremstyle{named}{}{}{\itshape}{}{\bfseries}{.}{.5em}{#1 \thmnote{#3}}
\theoremstyle{named}

\theoremstyle{definition}

\newtheorem{definition}[theorem]{Definition}
\newtheorem{warning}[theorem]{Warning}
\newtheorem{variant}[theorem]{Variant}

\newtheorem{example}[theorem]{Example}
\newtheorem*{example*}{Example}

\newtheorem*{question*}{Question}
\newtheorem{construction}[theorem]{Construction}

\newtheorem{remark}[theorem]{Remark}
\newtheorem{exercise}[theorem]{Exercise}

\newtheorem{numberedproof}[theorem]{Proof}

\begin{document}

\maketitle
\begin{abstract}
    \noindent
    We use derived methods to study the Gauss--Manin connection in Hochschild homology,
    infinitesimal cohomology, and derived de Rham
    cohomology. As applications, we give new approaches to nilinvariance,
    the Quillen spectral sequence, and the HKR filtration. We extend the results of Bhatt's work on
    de Rham cohomology in characteristic zero to infinitesimal cohomology in mixed characteristic
    and show that the comparison to Hartshorne's algebraic de Rham complex ``is'' the Gauss--Manin
    connection. Finally, we explain the main features of prismatic cohomology in characteristic zero via
    the Gauss--Manin connection.
\end{abstract}


\section{Introduction}\label{sec:intro}

The Gauss--Manin connection controls the behavior of de Rham cohomology in
families and is the source for many natural local
systems. The usual setting is a composable map $Y\xrightarrow{p} X\rightarrow S$ of
smooth morphisms; if $(E,\nabla)$ is an $S$-linear local system on $Y$, one
can push forward along $p$ to obtain a complex of $S$-linear local systems $(\R p_*E,\nabla)$ on
$X$. The Hodge filtration on the fibers of $\R p_*E$ is not respected by
$\nabla$, but it is only `one off', which is the statement of Griffiths transversality.

In characteristic $0$, the Gauss--Manin connection is especially important via
its role in variations of Hodge structures. In
characteristic $p$, the existence of extra flat sections thanks to
$\d(x^p)=0$ yields a
more flexible theory which is responsible for the nilpotence
results which apply to the monodromy of the Gauss--Manin connection, even in characteristic $0$.
These results are at the heart of Katz' work on
monodromy~\cite{katz-nilpotent} and $p$-curvature~\cite{katz-algebraic}.

The Gauss--Manin connection was apparently defined by Mumford in a
seminar at Harvard in 1966-67. It was extended by Katz in his
dissertation~\cite{katz-period} and then given its current form, which realizes the
connection as a differential in a spectral sequence associated to an explicit
filtration on the de Rham complex, by Katz and Oda~\cite{katz-oda}. The
transversality property was elaborated by Griffiths~\cite{griffiths-periods-2}
(see~\cite{griffiths-discussion} for further explanation).

The Katz--Oda approach is as follows. Consider a pair of smooth maps
$k\rightarrow R$ and $R\rightarrow S$ of commutative rings. One can filter the
classical de Rham complex $\Omega^\bullet_{S/k}$ by saying that $i$-forms $f\d
g_1\wedge\cdots\wedge \d g_i$ are in filtration degree $n$ if (locally) at
least $n$ of the $g_i$ are in $R$. Thus, if $R=k[x_1,\ldots,x_r]$ and
$S=R[y_1,\ldots,y_s]$, the $i$-form $\d x_1\wedge\cdots\wedge\d x_n\wedge\d
y_1\wedge\cdots\wedge\d y_{i-n}$ has filtration degree $n$ in
$\Omega^i_{S/k}$. This construction defines a finite
filtration on $\Omega^\bullet_{S/k}$ and the associated graded chain
complexes are $$\Omega^{\bullet-n}_{S/R}\otimes_R\Omega^n_{R/k}[-n].$$
In this paper, 
the Gauss--Manin connection simultaneously refers to this filtration on
$\Omega^\bullet_{S/k}$ and the associated coherent cochain complex\footnote{A coherent cochain complex as above consists of differentials $\nabla$ together with specified
nullhomotopies $H:\nabla^2\we 0$, specified compatibilities between the nullhomotopies
$H\circ\nabla$ and $\nabla\circ H$ expressing $\nabla^3\we 0$, and so on. The $\infty$-category of
such coherent cochain complexes with values in $\Mod_k$ is equivalent to $\widehat{\FMod}_k$, the
$\infty$-category of complete filtrations in $\Mod_k$. See~\cite{ariotta,raksit}.}
$$\Omega^\bullet_{S/R}\xrightarrow{\nabla}\Omega^{\bullet-1}_{S/R}\otimes_R\Omega^1_{R/k}\xrightarrow{\nabla}\Omega^{\bullet-2}_{S/R}\otimes_R\Omega^2_{R/k}\xrightarrow{\nabla}\cdots.$$
If we take $\gr^1$, we obtain the coherent chain complex $$\cdots\rightarrow 0\rightarrow\Omega^1_{S/R}\rightarrow
S\otimes_R\Omega^1_{R/k}[1]\rightarrow 0\rightarrow\cdots,$$
which corresponds to the extension
$S\otimes_R\Omega^1_{S/k}\rightarrow\Omega^1_{S/k}\rightarrow\Omega^1_{S/R}$,
i.e., the Kodaira--Spencer class
$\rho\in\Ext^1_S(\Omega^1_{S/R},S\otimes_R\Omega^1_{R/k})$.

One can globalize this construction. If $X\rightarrow\Spec R$ is smooth and
proper, then the Gauss--Manin connection on $\R\Gamma_{\dR}(X/R)$ gives rise to
a coherent cochain complex
$$\R\Gamma_\dR(X/R)\xrightarrow{\nabla}\R\Gamma_{\dR}(X/R)\otimes_R\Omega^1_{R/k}\xrightarrow{\nabla}\cdots.$$
Now, $\R\Gamma_{\dR}(X/R)$ is a perfect complex of $R$-modules. Taking
cohomology yields connections on $\H^i_{\dR}(X/R)$. If $R$ is a noetherian
$\bQ$-algebra, the
existence of the Gauss--Manin connection on each $\H^i_\dR(X/R)$ implies that
they are in fact finitely presented projective $R$-modules by a classical
argument in $D$-module theory. This is the crystalline nature of de Rham cohomology.\footnote{In
mixed characteristic it is no longer the case that the cohomology groups
$\H^i_{\dR}(X/R)$ are necessarily projective. Indeed, if $k=R=\bZ_2$ and $X$ is
an Enriques surface over $\Spec\bZ_2$, then
$\H^2_{\dR}(X/\bZ_2)\we\bZ_2^{10}\oplus\bZ/2$.
See~\cite[Prop.~7.3.5]{illusie-derham-witt} for the computation of the crystalline
cohomology of the special fiber of $X\rightarrow\Spec\bZ_p$ and use that this
agrees with the de Rham cohomology of $X$.}

The main goal of this paper is to introduce the Gauss--Manin connection for derived de Rham
cohomology. In fact, we make a construction which applies to a broad class of filtered
$\bE_\infty$-rings. Specializing to HKR-filtered Hochschild homology, Hodge-filtered infinitesimal
cohomology, and Hodge-filtered derived de Rham cohomology, we obtain the following result.

\begin{theorem}[Proof~\ref{proof:main}]\label{thm:main}
    Let $k\rightarrow R\rightarrow S$ be maps of derived commutative rings.
    \begin{enumerate}
        \item[{\em (a)}] If $k,R,S$ are connective, then there is a complete multiplicative
            $\bT_\fil$-equivariant
            filtration $\F^\star_\GM\HH_\fil(S/k)$ on
            $\HH_\fil(S/k)$ with associated graded pieces
            $\gr^i_\GM\HH_\fil(S/k)\we\HH_\fil(S/R)\otimes_R\ins^i\Lambda^i\L_{R/k}[i]$.
        \item[{\em (b)}] There is a complete multiplicative filtration $\F^\star_\GM\F^\star_\H\Infhat_{S/k}$
            on $\F^\star_\H\Infhat_{S/k}$ with associated graded pieces
            $\gr^i_\GM\F^\star_\H\Infhat_{S/k}\we\F^{\star}_\H\Infhat_{S/R}\widehat{\otimes}_R\ins^i\LSym^i(\L_{R/k}[-1])$.
        \item[{\em (c)}] There is a complete multiplicative filtration $\F^\star_\GM\F^\star_\H\dRhat_{S/k}$ on
            $\F^\star_\H\dRhat_{S/k}$ with associated graded pieces
            $\gr^i_\GM\F^\star_\H\dRhat_{S/k}\we\F^{\star}_\H\dRhat_{R/k}\widehat{\otimes}_R\ins^i\Lambda^i\L_{R/k}[-i]$.
    \end{enumerate}
\end{theorem}

We can unwind these Gauss--Manin connections and view them as coherent cochain complexes in
$\widehat{\FMod}_k$. For example, consider the case of $\F^\star_\GM\F^\star_\H\dRhat_{S/k}$.
The fact that $\ins^i\Lambda^i\L_{R/k}[-i]$, which is the complete exhaustive filtration on
$\Lambda^i\L_{R/k}[-i]$ corresponding to the coherent cochain complex with
$\Lambda^i\L_{R/k}$ in cohomological degree $i$ and $0$ elsewhere, is not concentrated in weight $0$ but rather in weight
$i$ is responsible for Griffiths transversality.
Indeed,
$$\F^\star_\H\dRhat_{S/R}\widehat{\otimes}_R\ins^i\L_{R/k}[-i]\we\F^{\star-i}_\H\dRhat_{S/R}\widehat{\otimes}_R\L_{R/k}[-i].$$
Thus, $\F^\star_\GM\F^\star_\H\dRhat_S/k$ corresponds to a coherent cochain complex of the form
$$\F^\star_\H\dRhat_{S/R}\xrightarrow{\nabla}\F^{\star-1}_\H\dRhat_{S/R}\widehat{\otimes}_R\L_{R/k}\xrightarrow{\nabla}\F^{\star-2}_\H\dRhat_{S/R}\widehat{\otimes}_R\Lambda^2\L_{R/k}\xrightarrow{\nabla}\cdots,$$
which is a derived version of the Gauss--Manin connection with Griffiths transversality.
In case $k\rightarrow R\rightarrow S$ are smooth maps of commutative rings, the Gauss--Manin
connection of Theorem~\ref{thm:main}(c) recovers the classical Katz--Oda complex.

The completions in (b) and (c) of Theorem~\ref{thm:main} are necessary.
However, if $\L_{R/k}$ is perfect, then there is no need to complete the tensor products.
In many cases, we will see that these completions give an Adams-type completion of
$\Lambda^i\L_{R/k}[-i]$ with respect to $S$, which plays an important role in our treatment of
derived de Rham cohomology in characteristic zero and infinitesimal cohomology in general.

On the other hand, under the connectivity assumption of (a), the HKR filtration is already
complete and this is preserved by tensoring with a bounded below complex.
There is a version of (a) without connectivity assumptions, but it would then apply only
to the HKR-completed Hochschild homology, which is not commonly considered.

The Gauss--Manin connection in Hochschild homology can be extended outside of the commutative
case. If $\Cscr$ is an $R$-linear dg category, then there is a filtration
$\F^\star_\GM\HH(\Cscr/k)$ with associated graded pieces
$\gr^i_\GM\HH(\Cscr/k)\we\HH(\Cscr/R)\otimes_R\Lambda^i\L_{R/k}[-i]$.
This filtration is complete under mild hypotheses, for example if $\Cscr$ is smooth and proper over
$R$. Using this, we obtain a connection on periodic cyclic homology which we expect to be related to
that of Connes, Getzler, and Goodwillie; see Theorem~\ref{thm:gmhp}. In our case, the flatness of this connection, which is
Getzler's contribution, comes for free. Stronger results, which say that the connection is
unipotent in a strong sense, have been obtained by
Kaledin~\cite{kaledin-nhdr,kaledin-spectral}, Mathew~\cite{mathew-kaledin}, and
Kaledin--Konovalov--Magidson~\cite{kaledin-konovalov-magidson}.

Our main application of the Gauss--Manin connection is the following strong form of nilinvariance
in infinitesimal cohomology, which uses our forthcoming joint work on descent with Iwasa and Krause.
This work is summarized in Section~\ref{sec:descent}.

\begin{theorem}[Theorem~\ref{thm:nilinvariance}]
    Let $k\rightarrow R\rightarrow S$ be maps of connective derived commutative rings. If
    $R\rightarrow S$ is a canonical cover such that $\pi_0R\rightarrow\pi_0S$ is surjective, then $\Infhat_{R/k}\rightarrow\Infhat_{S/k}$ is an
    equivalence.
\end{theorem}

As $\dRhat\we\Infhat$ in characteristic zero, this implies the corresponding result in derived de Rham
cohomology in characteristic zero, which is Kashiwara's Lemma. This in turn is related to Simpson's de Rham stack approach to
de Rham cohomology in characteristic zero~\cite{simpson_dr}. We will return to stacky connections
elsewhere. However, note that the theory is developed in great detail by Gaitsgory and Rozenblyum
in~\cite{gaitsgory-rozenblyum-crystals}. In their theory, the Gauss--Manin connection arises from the
shriek-pushforward of an induced map between de Rham stacks.

We use Theorem~\ref{thm:nilinvariance} to revisit Bhatt's results
from~\cite{bhatt-completions} and extend them to infinitesimal cohomology.
For example, we give a new proof of the comparison between Hodge-complete derived de Rham
cohomology and Hartshorne's algebraic de Rham cohomology in characteristic zero.
These mixed characteristic extensions were known to Bhatt via alternative methods and will
appear in his forthcoming joint work with Blickle, Schwede, and Tucker.

In characteristic $p$, we use nilinvariance to improve our results
from~\cite{antieau_crystallization}.

\begin{theorem}[Theorem~\ref{thm:charp}]
    Suppose that $k$ is a perfect field and $R$ is a connective derived commutative $k$-algebra
    such that $\pi_0k$ is finitely presented as a commutative $k$-algebra. Then, $\Infhat_{R/k}\we
    R^\perf$.
\end{theorem}

This was suggested to us a few years ago by Akhil Mathew and has been
obtained independently by Jiaqi Fu~\cite[Thm.~1.1]{fu_darboux} with somewhat stronger conditions on
$R$.

We also use the completeness results to give a conceptual construction of the Quillen spectral sequence
(see~\cite[Thm.~6.3]{quillen-cohomology} or~\cite[\href{https://stacks.math.columbia.edu/tag/08RC}{Tag
08RC}]{stacks-project})
and another definition of the HKR filtration.

\begin{theorem}[Theorem~\ref{thm:quillen}]\label{thm:quillen_intro}
    Suppose that $k\rightarrow R$ is a surjective map of commutative rings. The Quillen spectral
    sequence $$\E^1_{s,t}=\pi_{s+t}(\LSym^{-s}(\L_{R/k}))\Rightarrow\Tor_{s+t}^k(R,R)$$
    is the spectral sequence associated to the filtration $\F^\star_\H\Inf_{R/R\otimes_kR}$, which
    is complete.
\end{theorem}

In fact, $k$ and $R$ can be connective derived commutative rings such that
$\pi_0k\rightarrow\pi_0R$ is surjective in Theorem~\ref{thm:quillen_intro}.
The associated five-term exact sequence of low degree terms has applications in commutative
algebra; see for example~\cite{herzog}.

\begin{theorem}[Theorem~\ref{thm:hkr}]\label{thm:hkr_main}
    Let $k\rightarrow R$ be a map of derived commutative rings. The HKR filtration on
    $\HH(R/k)$ is equivalent to the Hodge filtration on the infinitesimal cohomology $\Inf_{R/\HH(R/k)}$.
\end{theorem}

The Hodge filtration on $\Inf_{R/\HH(R/k)}$ is naturally $S^1$-equivariant. Indeed, we have
$$\F^\star_\H\Inf_{R/\HH(R/k)}\we\HH(R/\F^\star_\H\Inf_{R/k}),$$
the Hochschild homology of $R$ relative to the augmentation $\F^\star_\H\Inf_{R/k}\rightarrow R$
computed in the symmetric monoidal category $\FMod_k$.
We do not see the filtered circle action from this perspective.

The To\"en--Vezzosi theory~\cite{tv-book} of derived foliations is built algebraically around filtered derived
commutative rings which look like Hodge-filtered derived de Rham cohomology or infinitesimal
cohomology. These foliations come in crystalline or infinitesimal flavors.
The proof of Theorem~\ref{thm:main} applies directly to these to yield Gauss--Manin connections 
associated to maps of derived foliations relative to $k$. We do not further explore the associated theory
here, but note that in the infinitesimal case Theorem~\ref{thm:hkr_main} provides a version of
Hochschild homology for derived foliations which looks worth pursuing further.

Finally, we explain how the $t$-de Rham comparison theorem in characteristic zero prismatic
cohomology of Wa{\ss}muth~\cite{wasmuth-thesis} is explained by
the Gauss--Manin connection.

\begin{theorem}[Theorem~\ref{thm:our_prism}]
    Let $k$ be a characteristic zero field and set $A=k\llbracket t\rrbracket$ and
    $\overline{A}=A/t\iso k$. Let $R$ be a smooth commutative $\overline{A}$-algebra, viewed as an
    $A$-algebra via restriction of scalars along $A\rightarrow\overline{A}$. Then, there is a
    natural equivalence $\R\Gamma((R/A)_\Prism,\Oscr)\we\dRhat_{R/A}$. The Gauss--Manin connection
    for the map $A\rightarrow R\llbracket t\rrbracket\rightarrow R$ recovers the $t$-de Rham
    comparison theorem.
\end{theorem}

\paragraph{Outline.} We prove symmetric monoidality and base change properties for $\HH$,
$\Inf$, and $\dR$ in Section~\ref{sec:sym}. We introduce a naive version of the Gauss--Manin
connection in Section~\ref{sec:horizontal}. This is enough to recover a connection on periodic cyclic
homology. In Section~\ref{sec:filtered_horizontal}, we give the general construction of the
Gauss--Manin connection and discuss Griffiths transversality. We give a universal property for the
Gauss--Manin connection in Section~\ref{sec:derived} and discuss descent properties of $\HH$,
$\Infhat$, and $\dRhat$ in Section~\ref{sec:descent}, based on joint work with Iwasa and Krause.
Our applications are given in Sections~\ref{sec:completions},~\ref{sec:quillen},~\ref{sec:hkr},
and~\ref{sec:prism}.

\paragraph{Notation.}
We adopt all of the notation and terminology from the
prequel~\cite{antieau_crystallization}.

\paragraph{Acknowledgments.} This paper was born out of a conversation
with Achim Krause in March 2020 which attempted to understand the approaches of
Krause--Nikolaus~\cite{krause-nikolaus} and
Liu--Wang~\cite{liu-wang} to absolute prismatic computations using descent from $\bS[z]$
as a Gauss--Manin connection. We thank
Bhargav Bhatt, Lukas Brantner, Rankeya Datta, Sanath Devalapurkar, Elden Elmanto, Jiaqi Fu, Haoyang Guo, Achim Krause,
Ryomei Iwasa, Srikanth Iyengar, Deven Manam, Akhil Mathew, Thomas Nikolaus, Joost Nuiten, Arpon
Raksit, Nick Rozenblyum, and Bertrand To\"en for discussions over the years.

This project was supported by NSF grants DMS-2120005, DMS-2102010, and
DMS-2152235, Simons Fellowships 666565 and 00005925, and the Simons Collaboration on Perfection.

\section{Symmetric monoidality of (co)homology}\label{sec:sym}

We begin by extending the universal properties of $\F^\star_\H\dR_{R/k}$ to a relative situation.
Below, $\FDAlg_{\F^\star_\H\dR_{R/k}}^\crys$ denotes the under category $(\FDAlg_k^\crys)_{\F^\star_\H\dR_{R/k}/}$
and similar for the completed variants.

\begin{proposition}\label{prop:gmraksit}
    Let $k$ be a derived commutative ring and let $R\in\DAlg_k$.
    There is a natural adjunction
    $$\F^\star_\H\dR_{(-)/k}\colon\DAlg_R\rightleftarrows\FDAlg_{\F^\star_\H\dR_{R/k}}^\crys\colon\gr^0,$$
    which upon completion gives a natural adjunction
    $$\F^\star_\H\dRhat_{(-)/k}\colon\DAlg_R\rightleftarrows\widehat{\FDAlg}_{\F^\star_\H\dRhat_{R/k}}^{\crys}\colon\gr^0.$$
\end{proposition}

\begin{proof}
    Raksit's theorem~\cite[Thm.~5.3.6]{raksit} says that there is an adjunction
    $$\F^\star_\H\widehat{\dR}_{(-)/k}\colon\DAlg_k\rightleftarrows\widehat{\FDAlg}^\crys_k\colon\gr^0.$$
    This result is decompleted in Corollary~\cite[Cor.~11.7]{antieau_crystallization} to obtain an adjunction
    $$\F^\star_\H\dR_{(-)/k}\colon\DAlg_k\rightleftarrows\FDAlg^\crys_k\colon\gr^0.$$
    It is a fact (see Exercise~\ref{exercise:comma}) that if
    $F\colon\Cscr\rightleftarrows\Dscr\colon G$ is an adjunction between
    $\infty$-categories and
    $c\in\Cscr$, then there is an induced adjunction
    $F'\colon\Cscr_{c/}\rightarrow\Dscr_{F(c)/}\colon G'$, where
    $F'(c\rightarrow d)\we(F(c)\rightarrow F(d))$ and $G'(F(c)\rightarrow e)\we
    (c\rightarrow GF(c)\rightarrow G(e))$, the first map being the unit map of
    the original adjunction. This general fact is applied to the
    adjunctions above and the object $R$ of $\DAlg_k$ using that by definition
    $(\DAlg_k)_{R/}\we\DAlg_R$ and similarly
    $(\FDAlg^{\crys}_k)_{\F^\star_\H\dR_{R/k}/}\we\FDAlg^\crys_{\F^\star_\H\dR_{R/k}}$.
    The completed case is similar.
\end{proof}

\begin{exercise}\label{exercise:comma}
    Prove the general fact about induced adjunctions for comma categories used
    in the proof of Proposition~\ref{prop:gmraksit} by
    using the identification of the mapping anima
    $\Map_{\Cscr_{c/}}(c\xrightarrow{\alpha} d,c\xrightarrow{\beta}
    e)$ with the fiber of $\Map_\Cscr(d,e)\xrightarrow{f\mapsto f\circ\alpha}\Map_{\Cscr}(c,e)$
    over $\beta$. See also~\cite[Prop.~5.2.5.1]{htt}.
\end{exercise}

\begin{remark}
    If $S$ is a derived commutative $R$-algebra, then the proposition gives a
    universal property for $\F^\star_\H\dR_{S/k}$ as a crystalline filtered derived
    commutative $\F^\star_\H\dR_{R/k}$-algebra: if $\F^\star T$ is a
    crystalline filtered derived commutative $\F^\star_\H\dR_{R/k}$-algebra, then
    the anima of maps $\F^\star_\H\dR_{S/k}\rightarrow\F^\star T$ is naturally equivalent to the
    anima of maps $S\rightarrow\gr^0T$ of derived commutative $R$-algebras.
\end{remark}

\begin{corollary}
    If $\F^\star T$ is a crystalline filtered derived commutative ring, then the left
    adjoint of $$\DAlg_{\gr^0T}\leftarrow\FDAlg^\crys_{\F^\star T}\colon\gr^0$$ is given by the
    composition
    $$\DAlg_{\gr^0T}\xrightarrow{\F^\star_\H\dR_{-/\gr^0T}}\FDAlg^\crys_{\F^\star_\H\dR_{\gr^0T/\bZ}}\xrightarrow{(-)\otimes_{\F^\star_\H\dR_{\gr^0T/\bZ}}\F^\star
    T}\FDAlg^\crys_{\F^\star T},$$
    where the right arrow is extension of scalars along the canonical arrow
    $\F^\star_\H\dR_{\gr^0T/\bZ}\rightarrow\F^\star T$ induced by the universal property of de
    Rham cohomology.
\end{corollary}

\begin{variant}
    There is an infinitesimal version of the story above; it is left to the reader to formulate it.
\end{variant}

\begin{variant}
    A similar result applies in Hochschild homology. Namely, given a derived commutative ring $k$
    and given $R\in\DAlg_k$, one obtains an adjunction
    $$\HH(-/k)\colon\DAlg_R\rightleftarrows (\cMod_{\bT^\vee}(\DAlg_k))_{\HH(R/k)/},$$
    where the right adjoint is obtained by forgetting the $S^1$-action and
    restricting along the map $R\rightarrow\HH(R/k)$ (implicitly chosen via
    the choice of a basepoint on $S^1$).
    This adjunction can be upgraded to an adjunction involving the HKR filtration
    $$\HH_\fil(-/k)\colon\DAlg_R\rightleftarrows(\cMod_{\bT_\fil^\vee}(\FDAlg_k^\inf))_{\HH_\fil(R/k)/}.$$
    The right adjoint is `forget the $\bT_\fil$-action and take $\F^0$'.
\end{variant}

Let $R\rightarrow S$ be a map in $\DAlg_k$. There is a well-known natural and
$\bT$-equivariant equivalence
$$\HH(S/k)\otimes_{\HH(R/k)}R\we\HH(S/R).$$ We call this the
\longdefidx{collapse formula}{Hochschild homology!collapse
formula} for Hochschild homology because it is obtained by
collapsing $S^1$ to a point in $\HH(R/k)\we{^{S^1}R}\rightarrow{^\ast R}\we R$
using the functoriality of tensors/copowers.

There are analogous statements for derived de Rham cohomology and its Hodge completion as well as for
infinitesimal cohomology and its Hodge completion.

\begin{proposition}[Collapse formula]\label{prop:reductioncollapse}
    Let $k\rightarrow R\rightarrow S$ be maps of derived commutative rings.
    There are natural equivalences
    \begin{itemize}
        \item   $\HH(S/k)\otimes_{\HH(R/k)}R\we\HH(S/R)$,
        \item   $\HH_\fil(S/k)\otimes_{\HH_\fil(R/k)}R\we\HH_\fil(S/R)$,
        \item   $\F^\star_\H\Inf_{S/k}\otimes_{\F^\star_\H\Inf_{R/k}}R\we\F^\star_\H\Inf_{S/R}$,
        \item $\F^\star_\H\Infhat_{S/k}\widehat{\otimes}_{\F^\star_\H\Infhat_{R/k}}R\we\F^\star_\H\Infhat_{S/R},$
        \item   $\F^\star_\H\dR_{S/k}\otimes_{\F^\star_\H\dR_{R/k}}R\we\F^\star_\H\dR_{S/R}$, and
        \item   $\F^\star_\H\dRhat_{S/k}\widehat{\otimes}_{\F^\star_\H\dRhat_{R/k}}R\we\F^\star_\H\dRhat_{S/R}$.
    \end{itemize}
\end{proposition}

In the tensor products above, $R$ stands for the filtered object $\ins^0R$.

\begin{proof}[Proof of Proposition~\ref{prop:reductioncollapse}]
    The completed statements for infinitesimal and derived de Rham cohomology follow from the uncompleted
    statements. We prove just the statement for derived de Rham cohomology and leave the other cases to the
    reader.
    Note that $$\F^\star_\H\dR_{S/k}\otimes_{\F^\star_\H\dR_{R/k}}R$$ is a
    crystalline filtered derived commutative ring with
    $$\gr^0\left(\F^\star_\H\dR_{S/k}\otimes_{\F^\star_\H\dR_{R/k}}R\right)\we S.$$
    Thus, there is a canonical map
    $$\F^\star_\H\dR_{S/R}\rightarrow\F^\star_\H\dR_{S/k}\otimes_{\F^\star_\H\dR_{R/k}}R.$$
    To show that
    this map is an equivalence, it is enough to check that the right-hand
    term has the correct universal property. Let $\F^\star T$ be a
    crystalline filtered derived commutative $R$-algebra. Then, there are equivalences
    $$\Map_R(\F^\star_\H\dR_{S/k}\otimes_{\F^\star_\H\dR_{R/k}}R,\F^\star
    T)\we\Map_{\F^\star_\H\dR_{R/k}}(\F^\star_\H\dR_{S/k},\F^\star
    T)\we\Map_R(S,\gr^0T)\we\Map_R(\F^\star_\H\dR_{S/R},\F^\star T),$$
    where the second equivalence is thanks to the adjunction
    of Proposition~\ref{prop:gmraksit} and the third is the universal property of derived
    de Rham cohomology.
\end{proof}

Proposition~\ref{prop:reductioncollapse} is a special case of the following more general result.

\begin{theorem}[Symmetric monoidality]
    Given a commutative diagram
    $$\xymatrix{
        k_1\ar[d]&k_0\ar[d]\ar[r]\ar[l]&k_2\ar[d]\\
        R_1&R_0\ar[r]\ar[l]&R_2,
    }$$
    in $\DAlg_\bZ$, the natural maps
    \begin{itemize}
        \item $\HH(R_1/k_1)\otimes_{\HH(R_0/k_0)}\HH(R_2/k_2)\rightarrow\HH(R_1\otimes_{R_0}R_2/k_1\otimes_{k_0}k_1)$,
        \item $\HH_\fil(R_1/k_1)\otimes_{\HH_\fil(R_0/k_0)}\HH_\fil(R_2/k_2)\rightarrow\HH_\fil(R_1\otimes_{R_0}R_2/k_1\otimes_{k_0}k_1)$,
        \item
            $\F^\star_\H\Inf_{R_1/k_1}\otimes_{\F^\star_\H\Inf_{R_0/k_0}}\F^\star_\H\Inf_{R_2/k_2}\rightarrow\F^\star_\H\Inf_{R_1\otimes_{R_0}R_2/k_1\otimes_{k_0}k_2}$,
        \item
            $\F^\star_\H\Infhat_{R_1/k_1}\widehat{\otimes}_{\F^\star_\H\Infhat_{R_0/k_0}}\F^\star_\H\Infhat_{R_2/k_2}\rightarrow\F^\star_\H\Infhat_{R_1\otimes_{R_0}R_2/k_1\otimes_{k_0}k_2}$,
        \item
            $\F^\star_\H\dR_{R_1/k_1}\otimes_{\F^\star_\H\dR_{R_0/k_0}}\F^\star_\H\dR_{R_2/k_2}\rightarrow\F^\star_\H\dR_{R_1\otimes_{R_0}R_2/k_1\otimes_{k_0}k_2}$,
            and
        \item
            $\F^\star_\H\dRhat_{R_1/k_1}\widehat{\otimes}_{\F^\star_\H\dRhat_{R_0/k_0}}\F^\star_\H\dRhat_{R_2/k_2}\rightarrow\F^\star_\H\dRhat_{R_1\otimes_{R_0}R_2/k_1\otimes_{k_0}k_2}$
    \end{itemize}
    are equivalences.
\end{theorem}

\begin{proof}
    One can argue as in the proof of Proposition~\ref{prop:reductioncollapse} to check that the
    universal property mapping out of the right-hand sides agrees with that of the left-hand sides
    in the relevant categories. Alternatively, given Proposition~\ref{prop:reductioncollapse}, we
    can argue as follows. We do the case of derived de Rham cohomology. Then, we can recognize the
    left-hand side as
    $$(\F^\star_\H\dR_{R_1/\bZ}\otimes_{\F^\star_\H\dR_{k_1/\bZ}}k_1)\otimes_{(\F^\star_\H\dR_{R_0/\bZ}\otimes_{\F^\star_\H\dR_{k_0/\bZ}}k_0)}(\F^\star_\H\dR_{R_2/\bZ}\otimes_{\F^\star_\H\dR_{k_2/\bZ}}k_2).$$
    By commuting colimits, this agrees with
    $$(\F^\star_\H\dR_{R_1/\bZ}\otimes_{\F^\star_\H\dR_{R_0/\bZ}}\F^\star_\H\dR_{R_2/\bZ})\otimes_{\F^\star_\H\dR_{k_1/\bZ}\otimes_{\F^\star_\H\dR_{k_0/\bZ}}\F^\star_\H\dR_{k_2/\bZ}}(k_1\otimes_{k_0}k_2),$$
    which agrees with the right-hand side using that $\F^\star_\H\dR_{-/\bZ}$ commutes with colimits
    and by another application of
    Proposition~\ref{prop:reductioncollapse}.
\end{proof}

\begin{example}
    Let $k\rightarrow R\rightarrow S$ be connective derived commutative rings. Because
    the objects involved are all connective in the Beilinson $t$-structure
    (see~\cite[Prop.~3.3.11]{raksit} or~\cite[Def.~6.3]{antieau_crystallization}), applying
    $\pi_0$ in the Beilinson $t$-structure to
    $\F^\star_\H\dR_{S/k}\otimes_{\F^\star_\H\dR_{R/k}}R\simeq\F^\star_\H\dR_{S/R}$
    gives an isomorphism
    \begin{equation}\label{eq:tensordiscrete}
        \Omega^\bullet_{\pi_0S/\pi_0k}\otimes_{\Omega^\bullet_{\pi_0R/\pi_0k}}\pi_0R\iso\Omega^\bullet_{\pi_0S/\pi_0R}
    \end{equation}
    of strict cdgas. Of course, this could have been proven directly from the
    universal property for the classical de Rham complex using the same proof as
    Proposition~\ref{prop:reductioncollapse}.
\end{example}

\begin{example}
    If $k\rightarrow R\rightarrow S$ are maps of commutative rings
    such that $\L_{R/k}$, $\L_{S/k}$, and $\L_{S/R}$ all have Tor-amplitude in
    $[0,0]$ over $R$, $S$, and $S$, respectively (for example if the morphisms
    are smooth),
    then the maps
    $\F^\star_\H\dRhat_{R/k}\rightarrow R$ and
    $\F^\star_\H\dRhat_{R/k}\rightarrow\F^\star_\H\dRhat_{S/k}$ are
    Tor-independent in the sense that the tensor product of classical de
    Rham complexes as in~\eqref{eq:tensordiscrete} is in fact derived (in
    complete filtered complexes).
\end{example}

Given $k\rightarrow R\rightarrow S$,
one might wonder whether the canonical map
$$\dRhat_{S/k}\otimes_{\dRhat_{R/k}}R\rightarrow\dRhat_{S/R}$$ is an
equivalence before completion of the tensor product. This problem turns out to be at the heart of
the way in which the Hodge filtration localizes de Rham theory.
In characteristic $0$, the map is not typically an equivalence before
completion. Here are two examples.

\begin{example}\label{ex:cm}
    Let $X$ be the elliptic curve over $\Gm$ defined
    by the Weierstrass equation $$y^2=x^3+\lambda.$$ The discriminant of the family is
    $-108\lambda^2$ so this Weierstrass equation is minimal and the singular fiber in a minimal
    family is a cusp (type II in Kodaira's classification). Let $Y$ be the complement
    of the $0$-section, so that $$Y=\Spec\bC[\lambda^{\pm
    1},x,y]/(y^2-x^3-\lambda x).$$ The smooth morphism $Y\rightarrow\Gm$ has fibers
    which are all elliptic curves minus a point. Thus, the fibers $F$ of $Y(\bC)\rightarrow\bC^\times$ 
    each have the homotopy type of $S^1\vee S^1$. The Leray--Serre spectral
    sequence
    $$\E_2^{s,t}=\H^s(\bC^\times,\H^t(F))\Rightarrow\H^{s+t}(Y(\bC),\bC)$$
    vanishes except for $(s,t)\in\{(0,0),(1,0)\}$. Indeed, the monodromy action
    of $\bZ=\pi_1\bC^\times$ on the local system $\H^1(F,\bC)$ is via the
    matrix $$\begin{pmatrix}1&1\\-1&0\end{pmatrix}$$ by~\cite[Table 6,
        p.159]{barth-peters-vandeven}. Thus,
    $$\H^0(\bC^\times,\H^1(F))=0=\H^1(\bC^\times,\H^1(F)).$$
    It follows from Bhatt's theorem~\cite[Thm.~4.10]{bhatt-completions} (or
    Theorem~\ref{thm:hartshorne} below) comparing Hodge-complete derived de Rham cohomology to
    Hartshorne's algebraic de Rham cohomology and Hartshorne's
    theorem~\cite[Thm.~IV.1.1]{hartshorne-derham} comparing algebraic de Rham
    cohomology to singular cohomology---although we only need the smooth case known to
    Grothendieck~\cite{grothendieck-derham}---that the pullback map $\dRhat_{\Gm/\bC}\rightarrow\dRhat_{Y/\bC}$ is an
    equivalence. Therefore, we see that
    $$\dRhat_{Y/\bC}\otimes_{\dRhat_{\Gm/\bC}}\bC[\lambda^{\pm
    1}]\simeq\bC[\lambda^{\pm 1}]\rightarrow\dRhat_{Y/\Gm}$$
    is not an equivalence since $\H^1(\dRhat_{Y/\Gm})$ is a rank $2$ vector
    bundle on $\Gm$.
\end{example}

\begin{example}
    Here is an algebraically easier example without the smoothness hypothesis.
    Suppose that $k$ is a commutative $\bQ$-algebra and consider the
    composition $k\rightarrow k[x]\rightarrow k$ where the second map sends $x$
    to $0$. Then,
    $\dRhat_{k[x]/k}\we k$ by the Poincar\'e lemma, $\dRhat_{k/k}\we k$, and
    $\dRhat_{k/k[x]}\we k\llbracket x\rrbracket$. However,
    $k\otimes_{\dRhat_{k[x]/k}}k[x]\we k[x]\rightarrow\dRhat_{k/k[x]}\we
    k\llbracket x\rrbracket$ is not an equivalence.
\end{example}

\section{The horizontal connection}\label{sec:horizontal}

If $\F^\star R$ is a filtered $\bE_\infty$-ring, then there is an adjunction
$$\F^\star_\h(-)\colon\Mod_{\F^0R}\rightleftarrows\FMod_{\F^\star R}\colon\F^0$$ where
the left adjoint is obtained by the \defidx{horizontal connection} functor $$M\mapsto
\F^\star_\h M=\F^\star R\otimes_{\ins^0\F^0(R)}\ins^0(M).$$ Here,
$\ins^0(\F^0R)\rightarrow\F^\star R$ is the counit of the adjunction
$$\ins^0\colon\Mod_\bZ\rightleftarrows\FMod_\bZ\colon\F^0$$ applied to
$\F^\star R$. Write $\overline{M}$ for $\gr^0_\h M\we\gr^0R\otimes_{\F^0R}M$.

Every horizontal connection $\F^\star_\h(M)$ satisfies the flatness condition
studied in our followup paper {\em F\&C IV}: $$\gr^iR\otimes_{\gr^0R}\gr^0_\h
M\we\gr^i_\h M.$$ However, and this is an absolutely crucial point, the flat
filtration is typically not complete. When it is, 
$M$ is the complex of
solutions to the ``differential equation''
$$\cdots\rightarrow\gr^{-1}R[-1]\otimes_{\gr^0R}\overline{M}\xrightarrow{\nabla}\overline{M}\xrightarrow{\nabla}\gr^1R[1]\otimes_{\gr^0R}\overline{M}\rightarrow\cdots.$$

\begin{example}[The de Rham case]
    If $k\rightarrow R\rightarrow S$ are smooth maps of commutative
    $\bF_p$-algebras, then applying the construction above to the
    $\dR_{R/k}$-module $\dR_{S/k}$ produces the classical Gauss--Manin
    connection on $\dR_{S/k}$. It is one important case
    where $\F^\star_\h\dR_{S/k}$ is already complete. The special thing about characteristic
    $p$ is the conjugate filtration whose existence shows that $\dR_{S/k}$ and $\dR_{S/R}$ are already Hodge-complete in the
    smooth case. This is the
    fundamental source of nilpotence in characteristic $p$.
\end{example}

\begin{warning}
    In characteristic $0$, even with smoothness hypotheses,
    the construction above does not produce the correct Gauss--Manin
    connection. Indeed, $\F^\star_\h\dR_{S/k}$ computed with respect to
    $\F^\star_\H\dR_{R/k}$ produces a typically incomplete filtration on
    $\dR_{S/k}\we k$, which is not so interesting, while
    $\F^\star_\h\dRhat_{S/k}$ computed with respect to
    $\F^\star_\H\dRhat_{R/k}$ is a filtration on $\dRhat_{S/k}$, as desired,
    but with $\gr^0_\h\dRhat_{S/k}$ generally impossible to understand and {\em not}
    typically equivalent to $\dRhat_{S/R}$; see Example~\ref{ex:cm}. To obtain the correct statement in
    characteristic $0$, one needs to work with two filtrations, which is the 
    source of non-unipotent monodromy. We do this in Section~\ref{sec:filtered_horizontal}.
\end{warning}

\begin{example}[The Higgs case]
    Let $R$ be a smooth commutative $k$-algebra.
    If we equip $\Omega^*_{R/k}$ with the trivial differential, then to a
    bundle over the shifted tangent bundle $\T_R[-1]$, i.e., a
    $\LSym_R(\Omega^1_{R/k}[1])\we|\Omega^*_{R/k}|$-module spectrum, we obtain a canonical Higgs bundle with
    vanishing Higgs field.
\end{example}

\begin{example}[The Bockstein case]
    If $R$ is a ring with $d\in R$, we can give $R$ the
    derived $d$-adic filtration $d^\star R$. If $M$ is an $R$-module, then the flat
    filtration on $M$ with respect to $d^\star R$ is the derived $d$-adic filtration. The associated
    spectral sequence is the $d$-based Bockstein spectral sequence. The flat
    filtration will be complete if and only if $M$ is derived $d$-complete to
    begin with.
\end{example}

\begin{example}[The Hochschild case]\label{ex:flathh}
    Let $k\rightarrow R\rightarrow S$ be maps of connective derived commutative rings. As
    $\HH(S/k)$ is a module spectrum over
    $\HH(R/k)$, it makes sense to speak of the flat filtration
    $\F^\star_\h\HH(S/k)$ with respect to $\HH_\fil(R/k)$. In this case,
    $\gr^0_\h\HH(S/k)\we R\otimes_{\HH(R/k)}\HH(S/k)\we\HH(S/R)$ by the
    collapse formula. Since $k$, $R$, and $S$ are
    connective, the induced
    filtration is complete because
    \begin{align*}
        \F^s(\HH_\fil(R/k)\otimes_{\ins^0(\HH(R/k))}\ins^0\HH(S/k))&\we\colim_{i+j\geq
        s}\F^i\HH_\fil(R/k)\otimes_{\HH(R/k)}\F^j\ins^0\HH(S/k)\\
        &\we\F^s\HH_\fil(R/k)\otimes_{\HH(R/k)}\HH(S/k),
    \end{align*} 
    which is $s$-connective by~\cite[Cor.~8.12]{antieau_crystallization}. It follows that the flat filtration is a complete
    exhaustive decreasing filtration on $\HH(S/k)$. The
    associated coherent cochain complex takes the form
    $$0\rightarrow\HH(S/R)\xrightarrow{\nabla}\L_{R/k}[2]\otimes_R\HH(S/R)\xrightarrow{\nabla}\Lambda^2\L_{R/k}[4]\otimes_R\HH(S/R)\rightarrow\cdots.$$
    Applying $(-)^{\t S^1}$ recovers (in the case where $k\rightarrow R$ is
    smooth) a connection which we expect agrees with Getzler's Gauss--Manin connection on periodic
    cyclic homology. This
    works because of the shear in Hochschild homology, which guarantees better
    completeness properties because of the increasing connectivity of the
    filtration; see Theorem~\ref{thm:gmhp} below.
\end{example}

\begin{remark}
    The connection of the previous example generalizes to the noncommutative setting if we replace
    $S$ by an
    $R$-linear dg category $\Cscr$ as long as either $R$ is smooth over $k$ or $\HH(\Cscr/k)$
    is bounded below. The latter happens for example if $\Cscr$ is smooth and proper
    over $R$ or $\Cscr\we\Perf(X)$ where $X$ is a quasi-compact quasi-separated
    $R$-scheme of finite Krull dimension.
\end{remark}

As explained to me by Getzler, the connection on $\HP(\Cscr/R)$ was known back
to work of Connes~\cite{connes} and Goodwillie~\cite{goodwillie_cyclic}. Getzler's contribution in~\cite{getzler} was
to prove the flatness of the connection on $\HP(\Cscr/R)$. See also Tsygan's paper~\cite{tsygan-gm}.
In the approach of our paper, the flatness of
the connection comes for free from the formalism of coherent cochain complexes.
This also answers the question of the possibility of making coherent choices
for the homotopies expressing flatness ``all the way up'', which was a question
of Deligne to Getzler and answered by Tsygan.

It seems to me to be unobserved previously that the connection on periodic
cyclic homology arises from a connection on Hochschild homology in the sense of
this paper.

\begin{theorem}[Gauss--Manin connection on periodic-cyclic homology]\label{thm:gmhp}
    Let $k\rightarrow R$ be a map of connective derived commutative rings and suppose
    that $\Cscr$ is an $R$-linear stable $\infty$-category (or dg category).
    There is a filtration $\F^\star_\h\HP(\Cscr/k)$ with
    $$\gr^s_\h\HP(\Cscr/k)\we\L_{R/k}[s]\otimes_R\HP(\Cscr/R).$$
    The filtration is complete if either $R$ is smooth over $k$ or
    $\HH(\Cscr/k)$ is bounded below.
\end{theorem}

\begin{proof}
    The filtration on $\HH(\Cscr/k)$ is constructed as in Example~\ref{ex:flathh}.
    It is complete if $R$ is smooth since in that case it is actually finite.
    It is complete if $\HH(\Cscr/k)$ is bounded above by the argument above.
    It follows that the induced filtration on $\HH(\Cscr/k)^{\h\bT}$ is
    complete in these cases. On the other hand, taking $\bT$-orbits does not
    decrease connectivity, so the induced filtration on $\HH(\Cscr/k)_{\h\bT}$ is also
    complete. Therefore, the induced filtration on $\HP(\Cscr/k)^{\t\bT}$ is
    complete.
\end{proof}

The following corollary applies in particular to smooth and proper $R$-linear categories $\Cscr$.

\begin{corollary}
    Suppose that $k$ is a field of characteristic $0$ and that $R$ is a smooth commutative
    $k$-algebra.
    If $\HH(\Cscr/R)$ is bounded with homology groups which are
    finitely presented over $R$, then each homology group of $\HP(\Cscr/R)$ is
    a finitely presented projective $R$-module.
\end{corollary}

\begin{proof}
    By Theorem~\ref{thm:gmhp}, $\HP(\Cscr/R)$ is equipped with a
    Gauss--Manin connection. As a coherent cochain complex, it takes the form
    $$\HP(\Cscr/R)\xrightarrow{\nabla}\HP(\Cscr/R)\otimes_R\L_{R/k}\xrightarrow{\nabla}\HP(\Cscr/R)\otimes_R\Lambda^2\L_{R/k}\xrightarrow{\nabla}\cdots.$$
    Since $R$ is smooth over $k$, each $\Lambda^i\L_{R/k}$ is concentrated in degree $0$.
    Hence, applying $\pi_n$ to the coherent cochain complex we find that
    $\HP_n(\Cscr/R)$ thus comes equipped with a flat connection. The finiteness
    assumption on $\HH_*(R/k)$ implies that $\HP_n(\Cscr/R)$ is finitely
    presented over $R$. Thus, each $\HP_n(\Cscr/R)$ is finitely presented
    projective, by~\cite[Prop.~8.8]{katz-nilpotent}.
\end{proof}

\begin{remark}
    In fact, more is true. Mathew shows in~\cite[Cor.~4.8]{mathew-kaledin} that if $\Cscr$ is
    smooth and proper over a commutative $\bQ$-algebra $R$, then $\HH(\Cscr/R)$ is a sum of
    shifts of finitely presented projective $R$-modules and the $\bT$-action on
    $\HH(\Cscr/R)$ is trivial, a very strong formality result. It implies in particular unipotence
    of the monodromy in the connection above. See work of
    Kaledin~\cite{kaledin-nhdr,kaledin-spectral},
    Kaledin--Konovalov--Magidson~\cite{kaledin-konovalov-magidson}, and Devalapurkar~\cite{devalapurkar-bp} for related
    results.
\end{remark}

\section{The filtered horizontal connection}\label{sec:filtered_horizontal}

We want to repeat the construction of the horizontal connection in a way that brings
a secondary filtration, typically the Hodge filtration, along for the ride and produces a connection which
satisfies Griffiths transversality. This allows us to localize appropriately to
capture more interesting differential equations.

\begin{definition}[Bifiltrations]
    Let $\BFMod_k$ denote the $\infty$-category
    $\Fun(\bZ^{\op}\times\bZ^{\op},\Mod_k)$ of double filtrations in $\Mod_k$.
    Day convolution endows $\BFMod_k$ with a symmetric monoidal structure arising from the symmetric
    monoidal structure on $\bZ^\op\times\bZ^\op$ given by addition in each coordinate.
    A general object of $\BFMod_k$ is denoted by $\F^{\star,\star}M$. We can construct horizontal and
    vertical associated graded filtered objects. The vertical associated graded filtrations
    $\gr^i\F^\star M=\cofib(\F^{i+1,\star}M\rightarrow\F^{i,\star}M)$ are the differences between the adjacent
    columns of the bifiltration, while the horizontal associated graded filtrations
    $\F^\star\gr^jM=\cofib(\F^{\star,j+1}M\rightarrow\F^{\star,j}M)$ are the differences between
    adjacent rows. We also have $\gr^{i,j}M$, which is the total cofiber of the square
    $$\xymatrix{
        \F^{i+1,j+1}M\ar[r]\ar[d]&\F^{i,j+1}M\ar[d]\\
        \F^{i+1,j}M\ar[r]&\F^{i,j}M.
    }$$ This total cofiber may be identified with $\gr^j(\gr^i\F^\star M)$ and $\gr^i(\F^\star\gr^j
    M)$.
\end{definition}

\begin{construction}
    There is a fully faithful diagonal functor $\bZ\rightarrow\bZ\times\bZ$ from which we obtain 
    an adjunction
    $$\Delta_!\colon\FMod_k\rightleftarrows\BFMod_k\colon\Delta^*$$
    where moreover $\Delta_!$ is symmetric monoidal. We
    have $$\F^{i,j}\Delta_!(\F^\star
    M)\we\colim_{\text{$(n,n)\rightarrow(i,j)$ in
    $\bZ^\op\times\bZ^\op$}}\F^nM\we\F^{\max(i,j)}M.$$
    From this, we see that $$\gr^{i,j}\Delta_!(\F^\star M)\we\begin{cases}
        \gr^i(\F^\star M)&\text{if $j=i$,}\\
        0&\text{otherwise.}
    \end{cases}$$
\end{construction}

\begin{remark}
    We find the three pictures starting on page~\pageref{fig:doublecomplexpicture} helpful for
    visualizing bifiltered complexes and specifically $\Delta_!(\F^\star M)$. The bifiltered picture of
    Figure~\ref{fig:bifilteredpicture} displays the filtration as a grid of commuting maps.
    The column cochain complex picture of Figure~\ref{fig:columnpicture} is obtained by looking at
    the coherent cochain complex in filtered complexes obtained by taking the vertical associated
    graded filtrations. Finally, the double
    complex picture of Figure~\ref{fig:doublecomplexpicture} takes associated graded pieces in both
    directions and assembles these into a coherent cochain complex in coherent cochain complexes (a
    coherent double complex up to signs).
\end{remark}

\begin{remark}
    The nature of coherent double complexes can be somewhat
    confusing at first. Note that we start with a boundary map
    $\gr^0M\rightarrow\gr^1M[1]$ in the coherent cochain complex
    associated to $\F^\star M$. But, when we write down the associated homotopy
    coherent double complex we see the square
    $$\xymatrix{
        \gr^0M\ar[r]\ar[d]&0\ar[d]\\
        0\ar[r]&\gr^1M[2].
    }$$
    The commutativity of the square implies that there is an induced map
    $\gr^0M\rightarrow\gr^1M[2][-1]$ which is our original differential.
\end{remark}

\begin{figure}
  \centering
        \centering
        $$\xymatrix{
            &\vdots\ar[d]&\vdots\ar[d]&\vdots\ar[d]&\\
            \cdots\ar[r]&\F^2M\ar[r]\ar[d]&\F^2M\ar[r]\ar[d]&\F^2M\ar[r]\ar[d]&\cdots\\
            \cdots\ar[r]&\F^2M\ar[r]\ar[d]&\F^1M\ar[r]\ar[d]&\F^1M\ar[r]\ar[d]&\cdots\\
            \cdots\ar[r]&\F^2M\ar[r]\ar[d]&\F^1M\ar[r]\ar[d]&\F^0M\ar[r]\ar[d]&\cdots\\
            &\vdots&\vdots&\vdots
        }$$
        \caption{The bifiltered picture for $\Delta_!(\F^\star M)$ in terms of the filtered pieces
        of $\F^\star M$. Here, we place
        $\F^{i,j}(\Delta_!(\F^\star M))$ in the $(-i,j)$-coordinate of the usual $xy$-plane.}
      \label{fig:bifilteredpicture}
  \centering
    $$\xymatrix{
        \cdots\ar[r]&\ins^0\gr^0(M)\ar[r]&\ins^1\gr^1(M)[1]\ar[r]&\ins^2\gr^2(M)[2]\ar[r]&\cdots
    }$$
    \caption{The column cochain complex picture for $\Delta_!(\F^\star
    M)$ in terms of the associated graded pieces of $\F^\star M$.}
    \label{fig:columnpicture}
  \centering
    $$\xymatrix{
        &\vdots\ar[d]&\vdots\ar[d]&\vdots\ar[d]&\\
        \cdots\ar[r]&\gr^0M\ar[r]\ar[d]&0\ar[r]\ar[d]&0\ar[r]\ar[d]&\cdots\\
        \cdots\ar[r]&0\ar[r]\ar[d]&\gr^1M[2]\ar[r]\ar[d]&0\ar[r]\ar[d]&\cdots\\
        \cdots\ar[r]&0\ar[r]\ar[d]&0\ar[r]\ar[d]&\gr^2M[4]\ar[r]\ar[d]&\cdots\\
        &\vdots&\vdots&\vdots
    }$$
    \caption{The double complex picture for $\Delta_!(\F^\star M)$ in terms of the associated
    graded pieces of $\F^\star M$.}
  \label{fig:doublecomplexpicture}
\end{figure}

\begin{construction}
    Now, we will restrict our attention to nonnegatively filtered objects, letting
    $$\F^+\Mod_k=\Fun(\bN^{\op},\Mod_k))\quad\text{and}\quad\BF^+\Mod_k=\Fun(\bN^{\op}\times\bN^{\op},\Mod_k).$$
    We still have a functor $\Delta_!\colon\F^+\Mod_k\rightarrow\BF^+\Mod_k$.
    We also have $$\chi\colon\bN\rightarrow\bN\times\bN$$ given by $$\chi(n)=(0,n).$$
    There is a natural transformation $\chi\rightarrow\Delta$ (this is why we need to
    restrict to nonnegatively filtered objects) from which we obtain a
    natural transformation $\chi_!\rightarrow\Delta_!$ of symmetric monoidal functors
    $\F^+\Mod_k\rightarrow\BF^+\Mod_k$.

    We have \begin{align*}
        \F^{i,j}\chi_!(\F^\star M)&\we\colim_{\text{$(0,n)\rightarrow(i,j)$ in $\bZ_{\geq 0}^\op\times\bZ_{\geq 0}^\op$}}\F^n M\\
        &\we\begin{cases}
            0&\text{if $i>0$ or}\\
            \F^jM&\text{if $i=0$.}
        \end{cases}
    \end{align*}
    In other words, $\chi_!\F^\star M$ is the bifiltered complex obtained by
    placing $\F^\star M$ in the $0$th column.
\end{construction}

\begin{definition}[The filtered horizontal connection]
    Let $\F^\star R$ be a nonnegatively
    filtered $\bE_\infty$-ring and let $\F^\star M$ be a nonnegatively
    filtered $\F^\star R$-module. We define the bifiltered object
    $$\F^{\star,\star}_\h(\F^\star M)\we\Delta_!\F^\star R\otimes_{\chi_!\F^\star
    R}\chi_!\F^\star M.$$ For short, we will write the filtered horizontal connection on $\F^\star
    M$ as $\F^{\star,\star}_\h M$. We let $\gr^i\F^\star M$, $\F^\star\gr^j M$, and $\gr^{i,j}_\h
    M$ denote the associated graded filtrations and pieces of the filtered horizontal connection.
\end{definition}

\begin{example}
    We consider the special case of maps $k\rightarrow R\rightarrow S$ of
    derived commutative rings and $\F^\star_\H\dR_{S/k}$ as a filtered
    $\F^\star_{R/k}\dR_{R/k}$-module. Let $\F^{\star,\star}_\h\dR_{S/k}$ be the filtered horizontal
    connection associated to $\F^\star_\H\dR_{S/k}$ as a module over $\F^\star_\H\dR_{R/k}$. We compute
    \begin{align*}
        \gr^{*,*}_\h\dR_{S/k}&\we\gr^{*,*}\Delta_!\F^\star_\H\dR_{R/k}\otimes_{\gr^{*,*}\chi_!\F^\star_\H\dR_{R/k}}\gr^{*,*}\chi_!\F^\star_\H\dR_{S/k}\\
        &\we\left(\bigoplus_{*\geq
        0}\Lambda^*\L_{R/k}[-*](*,*)\right)\otimes_{\left(\bigoplus_{*\geq
        0}\Lambda^*\L_{R/k}[-*](0,*)\right)}\left(\bigoplus_{*\geq
        0}\Lambda^*\L_{S/k}[-*](0,*)\right)\\
        &\we\left(\bigoplus_{*\geq
        0}\Lambda^*\L_{R/k}[-*](*,*)\right)\otimes_{R(0,0)}R(0,0)\otimes_{\left(\bigoplus_{*\geq
        0}\Lambda^*\L_{R/k}[-*](0,*)\right)}\left(\bigoplus_{*\geq
        0}\Lambda^*\L_{S/k}[-*](0,*)\right)\\
        &\we\left(\bigoplus_{*\geq
        0}\Lambda^*\L_{R/k}[-*](*,*)\right)\otimes_{R(0,0)}\left(\bigoplus_{*\geq
        0}\Lambda^*\L_{S/R}[-*](0,*)\right),
    \end{align*}
    where the third equivalence follows from the fact that the action of
    $\left(\bigoplus_{\star\geq 0}\Lambda^\star\L_{R/k}(0,\star)\right)$ on
    $\left(\bigoplus_{\star\geq
    0}\Lambda^\star\L_{R/k}(\star,\star)\right)$ factors through $R(0,0)$
    for degree reasons and the third equivalence follows from
    Proposition~\ref{prop:reductioncollapse}. We give the associated homotopy
    coherent double complex in Figure~\ref{fig:dc}.

    \begin{figure}
      \centering
        $$\xymatrix{
            &\vdots\ar[d]&\vdots\ar[d]&\vdots\ar[d]&\\
            \cdots\ar[r]&S\ar[r]\ar[d]&0\ar[r]\ar[d]&0\ar[r]\ar[d]&\cdots\\
            \cdots\ar[r]&\L_{S/R}\ar[r]\ar[d]&\L_{R/k}[1]\otimes_RS\ar[r]\ar[d]&0\ar[r]\ar[d]&\cdots\\
            \cdots\ar[r]&\Lambda^2\L_{S/R}\ar[r]\ar[d]&\L_{R/k}[1]\otimes_R\L_{S/R}\ar[r]\ar[d]&\Lambda^2\L_{R/k}[2]\otimes_RS\ar[r]\ar[d]&\cdots\\
            &\vdots&\vdots&\vdots
        }$$
        \caption{The double complex picture for the filtered flat filtration
        $\F^{\star,\star}_\h\dR_{S/k}$.}
      \label{fig:dc}
    \end{figure}

    Note that the `differential' $\L_{S/R}\rightarrow \L_{R/k}[1]\otimes_RS$ is
    the boundary map associated to the standard conormal sequence
    $$\L_{R/k}\otimes_R S\rightarrow\L_{S/k}\rightarrow\L_{S/R}.$$
    In general, the rows of Figure~\ref{fig:dc} are coherent cochain complexes corresponding to the
    standard filtration~\cite[Sec.~V.4]{illusie-cotangent-1} for computing
    $\Lambda^i\L_{S/k}$ from the conormal sequence.
\end{example}

We have already seen in Example~\ref{ex:cm} that completion is required in order to obtain useful
computational control via the Gauss--Manin connection.
In order to correctly capture the Gauss--Manin connection on $\dRhat_{S/k}$, we
must carefully complete the filtered horizontal connection while not altering the total complex $\dRhat_{S/k}$
itself. To do so, we will work in a vertically complete setting.

\begin{construction}[Column-complete bifiltrations]
    Thus, consider $$\Fun(\bN^\op,\widehat{\F^+\Mod}_k),$$
    a symmetric monoidal stable $\infty$-category via Day convolution, which uses addition in
    $\bN^\op$ and the completed tensor product on $\widehat{\F^+\Mod}_k$.
    We will write $\overline{\otimes}$ for this tensor product to remind the reader that it is
    completed only in the columns.
    The objects of $\Fun(\bN^\op,\widehat{\F^+\Mod}_k)$ are decreasing nonnegatively indexed filtrations of complete
    nonnegatively indexed filtered objects. Objects $\F^{\star,\star}M$ of $\Fun(\bN^\op,\widehat{\F^+\Mod}_k)\subseteq\BFMod_k$ will be referred to as
    column-complete nonnegative bifiltrations. The column-completeness means that for each $i\geq 0$ the
    limit $\lim_j \F^{i,j}M$ vanishes.
\end{construction}

\begin{example}
    If $\F^\star R$ is a complete filtered $\bE_\infty$-ring, then $\chi_!\F^\star R$ and
    $\Delta_!\F^\star R$ are column-complete nonnegative bifiltrations.
\end{example}

\begin{definition}[Column-complete filtered horizontal connection]
    Let $\F^\star R$ be a complete filtered $\bE_\infty$-ring and let $\F^\star M$ be a complete
    filtered $\F^\star R$-module. We define $$\F^{\star,\star}_{\hat{\h}} M=\Delta_!\F^\star
    R\overline{\otimes}_{\chi_!\F^\star R}\chi_!\F^\star M,$$
    the column-complete filtered horizontal connection.
\end{definition}

\begin{lemma}\label{lem:complete}
    Let $\F^\star R$ be a complete filtered $\bE_\infty$-ring and let $\F^\star
    M$ be a complete filtered $\F^\star R$-module. Then, the filtered
    horizontal connection $$\F^{\star,\star}_{\hat{\h}} M=\chi_!\F^\star M\overline{\otimes}_{\chi_!\F^\star R}\Delta_!\F^\star R$$
    is a complete filtration on $\F^\star M$ with graded pieces
    $$\gr^n\F^\star M\we\F^\star M\widehat{\otimes}_{\F^\star
    R}(\gr^0R\otimes_{\gr^0R}\ins^n\gr^nR).$$
\end{lemma}

\begin{proof}
    The claim about the graded pieces is clear since $\chi_!\F^\star M$ and
    $\chi_!\F^\star R$ are concentrated in the $0$ column. The content of
    the lemma is that $\F^{\star,\star}_{\hat{\h}} M$, viewed as an object of
    $\Fun(\bN^\op,\widehat{\F^+\Mod}_k)$, is complete. Because
    $\F^{n,\star}\chi_!\F^\star R\we 0$ for $n>0$ and similarly for
    $\chi_!\F^\star M$, we have $$\F^{n,\star}_{\hat{\h}}
    M\we\F^{0,\star}\chi_!M\widehat{\otimes}_{\F^{0,\star}\chi_!\F^\star R}\F^{n,\star}\Delta_!R\simeq\F^\star
    M\widehat{\otimes}_{\F^\star R}\F^{\star,\geq n}R.$$
    In particular, $\F^{0,\star}_{\hat{\h}}M\we\F^\star M$ so the result of the completeness claim will be a
    complete filtration on $\F^\star M$.
    Each of these pieces is complete, so the limit
    $$\lim_n\F^{n,\star}_{\hat{\h}} M$$ is complete. We compute the graded pieces
    $$\gr^m\left(\lim_n\F^{n,\star}_{\hat{\h}} M\right)\we\lim_n\gr^m(\F^{n,\star}_{\hat{\h}}
    M)\simeq\lim_n\gr^m(\F^\star M\widehat{\otimes}_{\F^\star R}\F^{\star\geq
    n}R).$$
    But, $\gr^m(\F^\star M\widehat{\otimes}_{\F^\star R}\F^{\star\geq
    n}R)$ vanishes for $n>m$ since everything is concentrated in nonnegative
    weights, so the limit is zero.
\end{proof}

The column-complete filtered horizontal connection allows us to establish the completeness of the
Gauss--Manin connection and the following theorem. The Griffiths transversality statement follows
for free from the fact that $\gr^i\F^\star\Delta_! M\we\ins^i\gr^iM$.

\begin{theorem}[The Gauss--Manin connection]\label{thm:griffiths}
    Let $k\rightarrow R\rightarrow S$ be maps of derived commutative rings.
    There is a complete multiplicative filtration
    $\F^\star_\GM\F^\star_\H\dRhat_{S/k}$ on $\F^\star_\H\dRhat_{S/k}$ with
    graded pieces $$\F^\star_\H\dRhat_{S/R}\widehat{\otimes}_R\ins^n\Lambda^n\L_{R/k}[-n].$$ We
    can view this as a coherent cochain complex (in complete filtered
    complexes)
    $$\F^\star_\H\dRhat_{S/R}\xrightarrow{\nabla}\F^{\star-1}_\H\dRhat_{S/R}\widehat{\otimes}_R\Lambda^1\L_{R/k}\xrightarrow{\nabla}\F^{\star-2}_\H\dRhat_{S/R}\widehat{\otimes}_R\Lambda^2\L_{R/k}\rightarrow\cdots.$$
\end{theorem}

\begin{proof}
    Set $\F^\star_\GM\F^\star_\H\dRhat_{S/k}=\F^{\star,\star}_{\hat{\h}}\dR_{S/k}$, the column-complete
    filtered horizontal connection associated to the $\F^\star_\H\dR_{R/k}$-module
    $\F^\star_\H\dR_{S/k}$. Completeness follows from Lemma~\ref{lem:complete}. The form of the
    associated coherent cochain complex follows from Proposition~\ref{prop:reductioncollapse}.
    Multiplicativity follows from the definition as a tensor product of bifiltered
    $\bE_\infty$-rings.
\end{proof}

\begin{numberedproof}[Proof of Theorem~\ref{thm:main}]\label{proof:main}
    Part (c) is Theorem~\ref{thm:griffiths}. The other parts follow from identical arguments.
    Completeness in part (a) follows as in Example~\ref{ex:flathh}. The $\bT_\fil$-equivariance in
    part (a) follows from Section~\ref{sec:derived}; specifically, the
    $\chi_!\bT_\fil^\vee$-coaction on
    $$\Delta_!\HH_\fil(\R/k)\otimes_{\chi_!\HH(R/k)}\chi_!\HH_\fil(S/k)$$ corresponds to a
    filtration on a filtered $k$-module with $\bT_\fil^\vee$-coaction. \qedsymbol{}
\end{numberedproof}

\begin{remark}[Griffiths transversality]
    We write down the filtered flat connection in a slightly different way to see Griffiths
    transversality in another way. We look at the column cochain complex picture. If we filter by
    columns, we get cochain complexes in filtered complexes:
    $$\gr^\star\F^\star(\Delta_!\F^\star R)\colon\cdots\rightarrow
    0\rightarrow\ins^0\gr^0R\rightarrow\ins^1\gr^1R\rightarrow\ins^2\gr^2\R\rightarrow\cdots$$
    and $$\gr^\star\F^\star(\chi_!\F^\star R)\colon\cdots\rightarrow
    0\rightarrow\F^\star R\rightarrow 0\rightarrow\cdots.$$
    In particular, we find that
    \begin{gather*}
        \gr^\star\F^\star(\Delta_!(\F^\star R)\otimes_{\chi_!(\F^\star R)}\chi_!(\F^\star M))\we\\
        \left(\cdots\rightarrow
        0\rightarrow (\ins^0\gr^0R)\otimes_{\F^\star R}\F^\star
        M\rightarrow(\ins^1\gr^1R)\otimes_{\F^\star R}\F^\star
        M\rightarrow(\ins^2\gr^2R)\otimes_{\F^\star R}\F^\star M\rightarrow\cdots\right),
    \end{gather*}
    which we may rewrite as $$\cdots\rightarrow
    0\rightarrow\F^\star\overline{M}\rightarrow\gr^1R\otimes_{\gr^0R}\F^{\star-1}
    \overline{M}\rightarrow\gr^2R\otimes_{\gr^0R}\F^{\star-2}\overline{M}\rightarrow\cdots,$$
    where $\F^\star\overline{M}$ is $\ins^0\gr^0R\otimes_{\F^\star R}\F^\star M$; this is
    \defidx{Griffiths transversality}.
\end{remark}

\begin{remark}
    For maps of connective derived commutative rings, part (c) of Theorem~\ref{thm:main} could be
    obtained by `deriving' the Katz--Oda connection. Parts (a) and (b) appear to be genuinely
    non-classical.
\end{remark}

\section{Derived algebraic structures}\label{sec:derived}

The following material will not be used elsewhere (except in the justification of the
$\bT_\fil$-equivariance in Theorem~\ref{thm:main}(a)); most readers should skip it.
The Gauss--Manin connection on derived de Rham cohomology (or infinitesimal cohomology,
Hochschild homology, or mixed homology) admits a natural universal property. As
with Raksit's use of crystalline filtered derived commutative algebras, it
seems cleanest to explain the universal property, and the associated notion of
derived bifiltered commutative ring, starting with the case of Hochschild
homology and then shearing down.

If $k$ is a commutative ring, then
on $\BFMod_k$ there is a derived algebraic context in the sense of~\cite{raksit} for which the
connective objects are those bifiltrations $\F^{\star,\star}M$ such
that $\F^{m,n}M$ is in $\Mod_k^\cn$ for each $(m,n)\in\bZ\times\bZ$. The
derived commutative algebras with respect to this context are called
bifiltered derived commutative $k$-algebras. Both
$\chi_!\colon\FMod_k\rightarrow\BFMod_k$ and
$\Delta_!\colon\FMod_k\rightarrow\BFMod_k$ are symmetric monoidal functors of the neutral derived
algebraic contexts and, when restricted to nonnegatively weighted objects, the
natural transformation $\chi_!\rightarrow\Delta_!$ is a natural transformation
of such functors.

It follows that $\chi_!\bT_\fil^\vee$ and $\Delta_!\bT_\fil^\vee$ are 
derived bicommutative bialgebras in $\BFMod_k$. Similarly, $\chi_!\HH_\fil(R/k)$
and $\Delta_!\HH_\fil(R/k)$ are derived commutative algebras in $\BFMod_k$ with
$\chi_!\bT_\fil^\vee$ and $\Delta_!\bT_\fil^\vee$-coactions. There is in fact a natural map
bicommutative bialgebras $$\Delta_!\bT_\fil^\vee\rightarrow\chi_!\bT_\fil^\vee.$$
There is in particular a natural $\chi_!\bT_\fil^\vee$-coaction on
$\Delta_!\HH_\fil(R/k)$ obtained by corestriction of scalars
and hence there is a $\chi_!\bT_\fil^\vee$-coaction on
$$\Delta_!\HH_\fil(R/k)\otimes_{\chi_!\HH_\fil(R/k)}\chi_!\HH_\fil(S/k).$$
This derived bifiltered commutative algebra with $\chi_!\bT_\fil^\vee$-action is
easily shown to admit a universal property.

\begin{lemma}
    The composite left adjoint
    $$\xymatrix{
        \DAlg_R\ar[r]^<<<<{\HH_\fil(-/k)}&(\cMod_{\bT^\vee_\fil}(\FDAlg_k))_{\HH_\fil(R/k)/}\ar[r]^<<<<{\chi_!}&
        (\cMod_{\chi_!\bT_\fil^\vee}(\mathrm{BFDAlg}_k))_{\chi_!\HH_\fil(R/k)/}\ar[d]^{\Delta_!\HH_\fil(R/k)\otimes_{\chi!\HH_\fil(R/k)}(-)}\\
        &&(\cMod_{\chi_!\bT_\fil^\vee}(\mathrm{BFDAlg}_k))_{\Delta_!\HH_\fil(R/k)/}
    }$$
    is $\Delta_!\HH_\fil(R/k)\otimes_{\chi_!\HH_\fil(R/k)}\chi_!\HH_\fil(S/k)$
    with the $\chi_!\bT_\fil$-action above.
\end{lemma}

Put another way, given a derived commutative $R$-algebra $S$, the data of a $\chi_!\bT_\fil$-equivariant derived bifiltered
$\Delta_!\HH_\fil(R/k)$-algebra $\F^{\star,\star}T$ and a
$\chi_!\bT_\fil$-equivariant map
$\Delta_!\HH_\fil(R/k)\otimes_{\chi_!\HH_\fil(R/k)}\chi_!\HH_\fil(S/k)\rightarrow\F^{\star\star}T$
is equivalent to the data of a derived commutative $R$-algebra map $S\rightarrow\F^{0,0}T$.

Now, note that $\F^\star\gr^*\chi_!\bT_\fil^\vee[-2*]$ is equivalent to the filtered graded object
$\ins^0\bD_-^\vee$. Thus, comodules over this object in $\F\Gr\Mod_k$ are equivalent to the full
subcategory of bifiltered objects which are column complete.

\begin{lemma}
    If $k\rightarrow R\rightarrow S$ are maps of derived commutative rings,
    then there is a natural equivalence
    $$\F^\star\gr^*\left(\Delta_!\HH_\fil(R/k)\otimes_{\chi_!\HH_\fil(R/k)}\HH_\fil(S/k)\right)[-2*]\we\Delta_!\F^\star_\H\dRhat_{R/k}\widehat{\otimes}_{\chi_!\F^\star_\H\dRhat_{R/k}}\chi_!\F^\star_\H\dRhat_{S/k}.$$
\end{lemma}

\begin{proof}
    In fact, $\F^\star\gr^*\chi_!\HH_\fil(R/k)[-2*]\we\chi_!\F^\star_\H\dRhat_{R/k}$ and
    similarly for $\chi_!\HH_\fil(S/k)$. By symmetric monoidality, it is enough to
    show that
    $\F^\star\gr^*\Delta_!\HH_\fil(R/k)[-2*]\we\Delta_!\dRhat_{R/k}$.
    The filtered coherent cochain complex of
    $\F^\star\gr^*\Delta_!\F^\star_\H\HH_\fil(R/k)[-2*]$
    is
    $$\ins^0R\rightarrow\ins^1\L_{R/k}\rightarrow\ins^2\Lambda^2\L_{R/k}\rightarrow\cdots,$$
    which is precisely $\Delta_!\dRhat_{R/k}$.
\end{proof}

\begin{definition}
    On $\F\Gr\Mod_k$ we can consider the monad which is infinitesimal filtered in the filtration
    direction and graded crystalline in the graded side. Let $\F\Gr\DAlg_k^{\inf,\crys}$ denote the
    $\infty$-category of left modules for the monad. The object $\ins^0\bD_-^\vee$ is an
    example of a cocommutative coalgebra in $\F\Gr\DAlg^{\inf,\crys}$. So, we can consider the
    $\infty$-category of comodules $\cMod_{\ins^0\bD_-^\vee}(\F\Gr\DAlg_k^{\inf,\crys})$. The
    objects of this category are some sort of column-complete bifiltered derived commutative ring.
    Examples include $\chi_!\F^\star\dRhat_{R/k}$ and $\Delta_!\F^\star\dRhat_{R/k}$.
\end{definition}

\begin{proposition}
    There is an adjunction
    $$\F^\star_\GM\F^\star_\H\dR_{-/k}\colon\DAlg_R\rightleftarrows(\cMod_{\ins^0\bD_-^\vee}(\F\Gr\DAlg_k^{\inf,\crys}))_{\Delta_!\F^\star\dRhat_{R/k}/}\colon\F^0\gr^0.$$
\end{proposition}

\begin{proof}
    We leave this for the reader.
\end{proof}

\section{Descent}\label{sec:descent}

In this short section we summarize our forthcoming work with Ryomei Iwasa and Achim Krause on descent for
invariants such as those considered in this paper; see also~\cite[Lem.~3.14]{akn-delta}. To this end, we consider the $\infty$-category
$\Pairs(\DAlg^\cn_\bZ)$ consisting of maps $A\rightarrow R$ of connective derived commutative
rings. We denote such a pair as $(R/A)$ for short.
We say that a collection $\{(R/A)\rightarrow(S_i/B_i)\}_{i\in I}$ is a canonical cover if
the maps $\{R\rightarrow S_i\}_{i\in I}$ define a canonical cover of $R$ in $\DAlg^\cn_\bZ$. Concretely, this
means that the diagram $$R\rightarrow\prod_{i\in I}S_i\stack{3}\prod_{i,j\in I}S_i\otimes_R
S_j\stack{5}\cdots$$
is a limit diagram and this remains true after base change.

\begin{theorem}\label{thm:descent}
    The following functors have descent for the canonical topology on $\Pairs(\DAlg^\cn_\bZ)$:
    \begin{enumerate}
        \item[{\em (a)}] $\HH(-/-)\colon\Pairs(\DAlg^\cn_\bZ)\rightarrow\DAlg_\bZ$;
        \item[{\em (b)}] $\HH_\fil(-/-)\colon\Pairs(\DAlg^\cn_\bZ)\rightarrow\FDAlg_\bZ$;
        \item[{\em (c)}] $\Infhat_{-/-}\colon\Pairs(\DAlg^\cn_\bZ)\rightarrow\DAlg_\bZ$;
        \item[{\em (d)}]
            $\F^\star_\H\Infhat_{-/-}\colon\Pairs(\DAlg^\cn_\bZ)\rightarrow\widehat{\FDAlg}_\bZ^\inf$;
        \item[{\em (e)}]
            $\dRhat_{-/-}\colon\Pairs(\DAlg^\cn_\bZ)\rightarrow\DAlg_\bZ$;
        \item[{\em (f)}]
            $\F^\star_\H\dRhat_{-/-}\colon\Pairs(\DAlg^\cn_\bZ)\rightarrow\widehat{\FDAlg}^\crys_\bZ$.
    \end{enumerate}
\end{theorem}

\begin{proof}[Sketch of proof]
    Here is a sketch in the case of $\HH$ and a single morphism $(R/A)\rightarrow(S/B)$
    defining a canonical cover. Consider the commutative square
    $$\xymatrix{
        \HH(R/A)\ar[r]\ar[d]&\HH(S/B)\ar[d]\\
        A\ar[r]&B.
    }$$
    In general, for connective pairs $(R/A)$, the map $\HH(R/A)\rightarrow R$ is a canonical cover
    since it has $1$-connective fiber. Thus, the composition $\HH(R/A)\rightarrow B$ is a canonical
    cover. It follows that $\HH(R/A)\rightarrow \HH(S/B)$ is a canonical cover by a standard fact
    about factorizations of effective epimorphisms in $\infty$-topoi; see~\cite[Cor.~6.2.3.12]{htt}.
    Now, symmetric monoidality of $\HH$ implies (a).
\end{proof}

\begin{example}\label{ex:diagonal}
    Consider the \v{C}ech complex of the map of pairs $(R/A)\rightarrow(R/R)$. Then, we find that
    $$\Infhat_{R/A}\we\Tot\left(\Infhat_{R/R}\stack{3}\Infhat_{R/R\otimes_AR}\stack{5}\cdots\right),$$
    which is familiar from de Rham cohomology in characteristic $0$.
\end{example}

\begin{example}[\v{C}ech--Alexander complexes]
    If $A\rightarrow R$ is a map of commutative rings and $S$ is a smooth commutative $A$-algebra
    with a map $S\rightarrow R$, then the \v{C}ech complex of the map $(R/A)\rightarrow(R/S)$
    gives rise to an equivalence
    $$\dRhat_{R/A}\we\Tot\left(\dRhat_{R/S}\stack{3}\dRhat_{R/S\otimes_AS}\stack{5}\cdots\right).$$
    The resulting cosimplicial objects are called \v{C}ech--Alexander complexes and
    Theorem~\ref{thm:descent} explains why they are such effective methods for computation in de
    Rham cohomology and related invariants.
\end{example}

\begin{corollary}
    Let $(R/A)\rightarrow(S/B)$ be a canonical cover of connective derived commutative rings with
    \v{C}ech complex $(S^\bullet/B^\bullet)$. If
    $M$ is a bounded below $A$-module, then the natural maps
    \begin{align*}
        M\otimes_A\Lambda^i\L_{R/A}&\rightarrow\Tot(M\otimes_A\Lambda^i\L_{S^\bullet/B^\bullet})\\
        M\otimes_A\LSym^i(\L_{R/A}[-1])&\rightarrow\Tot(M\otimes_A\LSym^i_{S^\bullet}(\Lambda^i\L_{S^\bullet/B^\bullet}[-1]))
    \end{align*}
    are equivalences.
\end{corollary}

\begin{proof}
    One can use base change along the map $A\rightarrow A\oplus M$ in $\widehat{\FDAlg}_A^\crys$ and the preservation of
    canonical covers under base change to see that
    $M\widehat{\otimes}_A\F^\star_\H\dRhat_{R/A}\we\Tot(M\widehat{\otimes}_A\F^\star_\H\dRhat_{S^\bullet/B^\bullet})$.
    Taking associated graded pieces establishes the first equivalence.
    The second follows by the same reasoning using infinitesimal cohomology.
\end{proof}

\section{Completions and derived infinitesimal cohomology}\label{sec:completions}

We show that the results of~\cite{bhatt-completions} extend to infinitesimal cohomology over an
arbitrary base. The questions we ask are ``when is the Hodge filtration complete'' and ``what is
the completion''. The strongest results use the Gauss--Manin connection in a crucial way.

Let $k\rightarrow R$ be a map of derived commutative rings.
We have observed already in the Poincar\'e lemma of~\cite[Prop.~9.6]{antieau_crystallization} that the natural map $k\rightarrow\Inf_{R/k}$ is an equivalence.
Now, we discuss the problem of showing that the Hodge filtration $\F^\star_\H\Inf_{R/k}$ is
complete so that $k\we\Inf_{R/k}\we\Infhat_{R/k}$. This is not always the case. For
example, $\Infhat_{\bQ[x^{\pm 1}]/\bQ}\we\dRhat_{\bQ[x^{\pm 1}]/\bQ}$ is not equivalent to $\bQ$ as
it contains a nonzero class $\tfrac{\d x}{x}$ in cohomological degree $1$.

\begin{lemma}\label{lem:connective_complete}
    If $k\rightarrow R$ is a map of connective derived commutative rings whose fiber $I$ is $1$-connective, then
    $\F^s_\H\Inf_{R/k}$ is $s$-connective for all $s\geq 0$. In particular, the
    Hodge filtration $\F^\star_\H\Inf_{R/k}$ is complete.
\end{lemma}

\begin{proof}
    The second statement follows from the first.
    We can build $R$ as a derived commutative $k$-algebra starting with $k$ by attaching only cells of dimension $2$ and higher.
    Since $s$-connective objects are closed under colimits, it will suffice to prove the result at
    each cell. The only difficult case is the first one. Let $T$ be the pushout
    $$\xymatrix{
        \LSym_k(I)\ar[r]\ar[d]&k\ar[d]\\
        k\ar[r]&T,
    }$$
    where one map is induced by the inclusion of $I$ as the fiber of $k\rightarrow R$ and the other
    is by the zero map. It will suffice to show that $\F^s\Inf_{T/k}$ is $s$-connective. We can
    test connectivity after base change along $k\rightarrow R$ since $I$ is $1$-connective. Thus,
    consider $$\xymatrix{\LSym_R(I\otimes_k R)\ar[r]\ar[d]&R\ar[d]\\R\ar[r]&T\otimes_kR}.$$ The map
    $I\otimes_kR\rightarrow R$ is now nullhomotopic (since $R\rightarrow R\otimes_kR$ has a
    retraction). Thus, we see that $T\otimes_kR$ is identified with $\LSym_R(I[1]\otimes_kR)$. This
    is a free algebra on a $2$-connective object. Let $T'=T\otimes_kR$ and $M=I[1]\otimes_kR$.
    Then, we have that $\F^\star\Inf_{T'/R}$
    is $\LSym^\inf_R(\fib(\ins^0 M\rightarrow\ins^1 M))$. Each $\LSym^s(\fib(\ins^0
    M\rightarrow\ins^1 M))$ has a finite filtration with associated graded pieces given by
    $\LSym^a(\ins^1 M[-1])\otimes\LSym^b(\ins^0 M)$ for $a+b=s$. Taking  $\F^r$ of this term yields
    zero if $r>a$ and otherwise yields $\LSym^a(M[-1])\otimes\LSym^b(M)$ for $r\leq a$. This is
    $a$-connective since $M[-1]$ is $1$-connective. This completes the proof for the first cell
    attachment and the later ones are similar.
\end{proof}

\begin{construction}[Completions]\label{const:completions}
    There are several notions of completion in higher algebra.
    See~\cite{carlsson,DwGr,greenlees-may,MNN17} and~\cite[Chap.~7]{sag}.
    Let $k\rightarrow R$ be a map of $\bE_\infty$-rings with \v{C}ech complex $R^\bullet$. Let
    $k_R^\wedge=\Tot(R^\bullet)$, viewed as an $\bE_\infty$-ring. If $k\rightarrow R$ is a map of
    derived commutative rings, then we can and do view $k_R^\wedge$ as a derived commutative ring.
    We say that $k$ is $R$-complete if the natural map $k\rightarrow k_R^\wedge$ is an equivalence.
    Similarly, if $M$ is a $k$-module, let $M_R^\wedge=\Tot(M\otimes_kR^\bullet)$ and we say that
    $M$ is $R$-complete if $M\rightarrow M_R^\wedge$ is an equivalence.

    On the other hand, we can consider the class $S$ of morphisms $M\rightarrow N$ in $\Mod_k$ such
    that $R\otimes_kM\rightarrow R\otimes_kN$ is an equivalence. Then, the localization
    $\Mod_k[S^{-1}]$ exists and the localization functor $\Mod_k\rightarrow\Mod_k[S^{-1}]$ admits a
    fully faithful right adjoint. We identify $\Mod_k[S^{-1}]$ with its essential image
    $(\Mod_k)_R^\wedge$ in $\Mod_k$ and call the $k$-modules in the image categorically $R$-complete.
    See~\cite[Def.~2.19]{MNN17}.

    An $R$-complete module $M$ is automatically categorically $R$-complete since (i) $R$-modules
    are categorically $R$-complete and (ii) the categorically $R$-complete
    $k$-modules are closed under limits in $\Mod_k$.
    When $R$ is dualizable as a $k$-module, then $M\in\Mod_k$ is $R$-complete if and only if it is
    categorically $R$-complete~\cite[Prop.~2.21]{MNN17} and the categorical completion functor
    $\Mod_k\rightarrow(\Mod_k)_R^\wedge$ is given by $M\mapsto M_R^\wedge$.
\end{construction}

\begin{lemma}\label{lem:adams_completion}
    Suppose that $k\rightarrow R$ is a map of connective derived commutative rings inducing a surjection
    $\pi_0$.
    \begin{enumerate}
        \item[{\em (a)}] There is a natural equivalence $k_R^\wedge\we\Infhat_{R/k}$.
        \item[{\em (b)}] If $M$ is a bounded below $k$-module, then the natural map
            $M_R^\wedge\rightarrow\Infhat_{R/k}\widehat{\otimes}_kM$ is an equivalence.
    \end{enumerate}
\end{lemma}

\begin{proof}
    Consider the \v{C}ech complex $R^\bullet$ of $k\rightarrow R$. There is a natural augmentation
    $R^\bullet\rightarrow R$ given by multiplication. The induced map
    $M\widehat{\otimes}_k\Infhat_{R/k}\rightarrow\Tot(M\widehat{\otimes}_k\Infhat_{R/R^\bullet})$ is an
    equivalence by Theorem~\ref{thm:descent}.
    Since $\pi_0k\rightarrow\pi_0R$ is surjective, each $R^n\rightarrow R$ is an isomorphism on $\pi_0$
    and as the multiplication maps are split, it is a surjection on $\pi_1$. Thus, each
    map $R^n\rightarrow R$ satisfies the hypotheses of Lemma~\ref{lem:connective_complete}, so
    $M\widehat{\otimes}_k\Infhat_{R/R^\bullet}\we (M\otimes_kR^\bullet)$. Thus,
    $M\widehat{\otimes}_k\Infhat_{R/k}\we\Tot(M\widehat{\otimes}_k\Infhat_{R/R^\bullet})\we\Tot(M\otimes_k R^\bullet)\we M_R^\wedge$.
\end{proof}

\begin{example}
    Suppose that $k\rightarrow R$ is a map of connective derived commutative rings such that
    $\pi_0k\rightarrow\pi_0R$ is an isomorphism. Then, $k\we k_R^\wedge\we\Infhat_{R/k}$. The first
    equivalence follows from the fact that under the hypothesis $k\rightarrow R$ is a canonical
    cover as in the sketch of the proof of Theorem~\ref{thm:descent}. The second equivalence
    follows from Lemma~\ref{lem:adams_completion}(a).
\end{example}

Now, we can prove a strong form of nilinvariance, or Kashiwara's Lemma.

\begin{theorem}[Canonical invariance]\label{thm:nilinvariance}
    Let $k\rightarrow R\rightarrow S$ be maps of connective derived commutative rings.
    If $R\rightarrow S$ is a canonical cover which induces a surjection on $\pi_0$,
    then $\Infhat_{R/k}\we\Infhat_{S/k}$.
\end{theorem}

\begin{proof}
    Consider the Gauss--Manin filtration on $\Infhat_{S/k}$. This is a complete filtration with
    associated graded pieces given by $$\Infhat_{S/R}\widehat{\otimes}_R\LSym^i(\L_{R/k}[-1]).$$
    Since $\pi_0R\rightarrow\pi_0S$ is surjective, we have that
    $\Infhat_{S/R}\widehat{\otimes}_R\LSym^i(\L_{R/k}[-1])$ is equivalent to the $S$-completion of
    $\LSym^i(\L_{R/k}[-1])$ by Lemma~\ref{lem:adams_completion}. However, as $\LSym^i(\L_{R/k}[-1])$ is bounded below and $R\rightarrow
    S$ is a canonical cover, the natural map
    $\LSym^i(\L_{R/k}[-1])\rightarrow(\LSym^i(\L_{R/k}[-1]))_S^\wedge$ is an equivalence. It
    follows that the Gauss--Manin filtration on $\Infhat_{S/k}$ agrees with the Hodge filtration on
    $\Infhat_{R/k}$.
\end{proof}

\begin{corollary}
    Let $k$ be a connective derived commutative ring and let $I\subseteq\pi_0k$ be an ideal
    generated by elements $x_1,\ldots,x_n$. Let $R$ be the Koszul complex of $(x_1,\ldots,x_n)$ and
    let $S$ denote the commutative ring $(\pi_0k)/I$. Then, $k_R^\wedge\we k_S^\wedge$.
\end{corollary}

\begin{proof}
    This follows from Theorem~\ref{thm:nilinvariance} since $R\rightarrow S$ is a canonical cover.
\end{proof}

\begin{corollary}\label{cor:complete_omni}
    Let $k$ be a connective derived commutative ring and let $I\subseteq\pi_0k$ be an ideal
    generated by elements $x_1,\ldots,x_n$. Let $R$ be the Koszul complex of $(x_1,\ldots,x_n)$ and
    let $S$ denote the commutative ring $(\pi_0k)/I$. The following conditions are
    equivalent:
    \begin{enumerate}
        \item[{\em (a)}] $k$ is $R$-complete;
        \item[{\em (b)}] the Hodge filtration $\F^\star_\H\Inf_{R/k}$ is complete;
        \item[{\em (c)}] $k$ is $S$-complete;
        \item[{\em (d)}] the Hodge filtration $\F^\star_\H\Inf_{S/k}$ is complete;
        \item[{\em (e)}] $k$ is categorically $R$-complete.
    \end{enumerate}
\end{corollary}

\begin{proof}
    As $\pi_0k\rightarrow\pi_0R$ and $\pi_0k\rightarrow\pi_0S$ are surjective,
    Lemma~\ref{lem:adams_completion} implies that (a) is equivalent to (b) and (c) is equivalent to
    (d). However, by Lemma~\ref{lem:adams_completion} and Theorem~\ref{thm:nilinvariance}, we
    find that $k_R^\wedge\we k_S^\wedge$, so (a) is
    equivalent to (c). Since the Koszul complex $R$ is dualizable as a $k$-module, the techniques
    of~\cite[Prop.~2.21]{MNN17} apply to show that the categorical completion of $k$ is computed by the
    completion; hence (a) is equivalent to (e).
\end{proof}

We have one more ingredient, a comparison between the classical $I$-adic completion and the
completion with respect to $R/I$.

\begin{lemma}\label{lem:classical}
    Suppose that $R$ is a noetherian commutative ring and that $I\subseteq R$ is an ideal. Let
    $S=R/I$. If $M$ is a finitely generated $R$-module, then the classical $I$-adic completion $\lim_n M/I^n$ of $M$ agrees with the
    $S$-completion $M_S^\wedge$.
\end{lemma}

\begin{proof}
    See~\cite[Prop.~3.15]{bhatt-completions} or~\cite[Thm.~4.4]{carlsson} or~\cite[Prop.~7.3.6.1]{sag}.
\end{proof}

As a consequence, we can strengthen our results from~\cite[Thm.~10.3]{antieau_crystallization} on
Hodge-complete infinitesimal cohomology in characteristic $p$ by removing the quasisyntomic
hypothesis at the cost of adding some finiteness hypothesis. This result was indicated to me
several years ago by Akhil Mathew and has also been obtained by Jiaqi Fu in~\cite[Thm.~1.1]{fu_darboux} with
stronger finiteness hypotheses.

\begin{theorem}\label{thm:charp}
    Suppose that $k$ is a perfect field of characteristic $p>0$.
    Let $R$ be a connective derived commutative $k$-algebra such that $\pi_0R$ is finitely
    presented as a commutative $k$-algebra. Then, the maps in the natural commutative diagram
    $$\xymatrix{
        R^\perf\ar[r]\ar[d]&(\pi_0R)^\perf\ar[d]\\
        \Infhat_{R/k}\ar[r]&\Infhat_{(\pi_0R)/k}
    }$$ are equivalences.
\end{theorem}

We thank Rankeya Datta for help with the commutative algebra in the following proof.

\begin{proof}[Proof of Theorem~\ref{thm:charp}]
    The top map is an equivalence as perfect derived commutative $\bF_p$-algebras are
    coconnective by~\cite[Prop.~11.6]{bhatt-scholze-witt}. The bottom map is an equivalence by
    Theorem~\ref{thm:nilinvariance}. Thus, it suffices to see that the right map is an
    equivalence. So, now assume that $R\we\pi_0R$.
    Let $S\rightarrow R$ be a surjection from a smooth commutative $k$-algebra with kernel $I$.
    Let $S_I^\wedge$ denote the $I$-adic completion of $S$ (equivalent to the completion
    $S_R^\wedge$ by Lemma~\ref{lem:classical}). Since $S$ is smooth over $k$, it is in particular a
    $G$-ring. By a theorem of
    Nagata, it follows that $S\rightarrow S_I^\wedge$ is geometrically regular;
    see~\cite[Thm.~79]{matsumura} or~\cite[\href{https://stacks.math.columbia.edu/tag/0AH2}{Tag
    0AH2}]{stacks-project}. By N\'eron--Popescu desingularization~\cite{popescu,swan}, it
    follows that $S\rightarrow S_I^\wedge$ is ind-smooth. In particular, $\L_{S_I^\wedge/S}$ is
    equivalent to a
    flat $S_I^\wedge$-module concentrated in degree $0$. The conormal sequence now implies that
    $\L_{S_I^\wedge/k}$ is a flat-module concentrated in degree $0$. It follows that $S_I^\wedge$
    is a quasisyntomic $k$-algebra.
    By~\cite[Thm.~10.3]{antieau_crystallization}, it follows that
    $(S_I^\wedge)^\perf\we\Infhat_{S_I^\wedge/k}$. However, $S_I^\wedge$ is also $F$-finite,
    so in fact $\L_{S_I^\wedge/k}$ is a finitely presented projective $S_I^\wedge$-module
    by~\cite[Lem.~3.10]{dundas-morrow}; in
    particular, it is perfect. Now, the Gauss--Manin connection on $\Infhat_{R/k}$ associated to the composition
    $k\rightarrow S_I^\wedge\rightarrow R$ takes the form
    $$\Infhat_{R/S_I^\wedge}\rightarrow\Infhat_{R/S_I^\wedge}\widehat{\otimes}_{S_I^\wedge}\L_{S_I^\wedge/k}\rightarrow\cdots.$$
    Each $\LSym^i(\L_{S_I^\wedge/k}[-1])$ is perfect over $S_I^\wedge$,
    so the completions are not necessary in the tensor products. But, $S_I^\wedge$ is $R$-complete, so
    $S_I^\wedge\we\Infhat_{R/S_I^\wedge}$ by Lemma~\ref{lem:adams_completion}. It follows that
    the maps $(S_I^\wedge)^\perf\rightarrow\Infhat_{S_I^\wedge/k}\rightarrow\Infhat_{R/k}$ are
    equivalences. Finally, $(S_I^\wedge)^\perf\we\lim(S/I^n)^\perf$, but this limit diagram is constant on
    $R^\perf$.
\end{proof}

Finally, we use the Gauss--Manin connection to prove Bhatt's comparison theorem between derived de
Rham cohomology and Hartshorne's algebraic de Rham cohomology.

\begin{theorem}\label{thm:hartshorne}
    Let $k\rightarrow R\rightarrow S$ be maps of commutative $\bQ$-algebras. Assume that $R$
    is smooth over $k$ and $R\rightarrow S$ is surjective with kernel $I$. Then, $\dRhat_{S/k}$ is equivalent to
    the classically $I$-adically completed de Rham complex
    $$R_I^\wedge\rightarrow(\Omega^1_{R/k})_I^\wedge\rightarrow(\Omega^2_{R/k})_I^\wedge\rightarrow\cdots.$$
\end{theorem}

\begin{proof}
    Consider the Gauss--Manin connection
    $$\dRhat_{S/R}\rightarrow\dRhat_{S/R}\widehat{\otimes}_R\Omega^1_{R/k}\rightarrow\cdots,$$
    which computes $\dRhat_{S/k}$. By Corollary~\ref{cor:complete_omni} and
    Lemma~\ref{lem:classical}, we have that $\dRhat_{S/R}\we R_I^\wedge$ and similarly
    $\dRhat_{S/R}\widehat{\otimes}_R\Omega^i_{R/k}\we (\Omega^i_{R/k})_I^\wedge$ for all $i\geq 0$
    since $\Omega^i_{R/k}$ is finitely generated.
\end{proof}

\begin{remark}
    Note in the context of Theorem~\ref{thm:hartshorne} that the associated graded pieces of the
    Hodge filtration on $\dRhat_{R/k}$ and $\dRhat_{S/k}$ are typically not equivalent.
\end{remark}

In characteristic zero we also obtain $\bA^1$-invariance.

\begin{proposition}[$\bA^1$-invariance]
    Let $k$ be a connective derived commutative $\bQ$-algebra and let $R$ be a connective derived
    commutative $k$-algebra. Then, the natural map $\Infhat_{R/k}\rightarrow\Infhat_{R[t]/k}$ is an
    equivalence.
\end{proposition}

\begin{proof}
    Arguing as in the proof of Theorem~\ref{thm:nilinvariance}, it is enough to
    see that the natural map
    $\LSym^i_R(\L_{R/k}[-1])\rightarrow\Infhat_{R[t]/R}\widehat{\otimes}_R\LSym^i_R(\L_{R/k}[-1])$
    is an equivalence for $i\geq 0$.
    The right-hand term is the fiber of $$R[t]\xrightarrow{\d} R[t]\dt$$ tensored with
    $\LSym_R^i(\L_{R/k}[-1])$. But, this fiber is $R$ by the classical Poincar\'e lemma.
\end{proof}

\section{The Quillen spectral sequence}\label{sec:quillen}

The Quillen spectral sequence admits an exotic-looking $\E^1$-page involving symmetric powers of
the cotangent complex and converging to $\Tor^k_*(R,R)$. It was introduced
in~\cite[Thm.~6.3]{quillen-cohomology}; see~\cite[\href{https://stacks.math.columbia.edu/tag/08RC}{Tag
08RC}]{stacks-project}. We give an easy explanation for the
existence of this spectral sequence using infinitesimal cohomology. The non-trivial work then lies in establishing
convergence. A similar perspective is contained in Bhatt's paper~\cite{bhatt-completions} but is
restricted to characteristic $0$. One of our main points is that infinitesimal cohomology in many
ways extends the usual results about derived de Rham cohomology in characteristic $0$ to mixed
characteristic. The Poincar\'e lemma for infinitesimal
cohomology~\cite[Prop.~9.6]{antieau_crystallization}, which we use in the proof below, was another instance of this principle.

\begin{theorem}[Quillen spectral sequence]\label{thm:quillen}
    Let $k\rightarrow R$ be a map of derived commutative rings.
    The Quillen spectral sequence is the spectral sequence associated to the Hodge filtration
    $\F^\star_\H\Inf_{R/R\otimes_kR}$ on
    $\Inf_{R/R\otimes_kR}$. This filtration is complete, and the spectral sequence converges, if
    $k$ and $R$ are connective and $\pi_0k\rightarrow\pi_0R$ is surjective.
\end{theorem}

\begin{proof}
    Note that $\L_{R/R\otimes_kR}\we\L_{R/k}[1]$. Thus, the associated graded pieces of the Hodge
    filtration on $\Inf_{R/R\otimes_kR}$ are given by $\LSym_R^*(\L_{R/k})$.
    Since $\Inf_{R/R\otimes_kR}\we R\otimes_kR$ by the Poincar\'e
    lemma~\cite[Prop.~9.6]{antieau_crystallization},
    the associated Hodge--infinitesimal spectral sequence takes the form
    $$\E^1_{s,t}=\pi_{s+t}(\LSym^{-s}(\L_{R/k}))\Rightarrow\Tor_{s+t}^k(R,R)$$
    with differential $\d^r$ of bidegree $(-r,r-1)$. (This is a standard reindexing of the spectral sequence
    of~\cite{quillen-cohomology}.) The claim about convergence follows from
    Lemma~\ref{lem:connective_complete}.
\end{proof}

\begin{variant}
    Suppose now that $R,S,T$ are derived commutative $k$-algebras and that there is a map
    $R\otimes_kS\rightarrow T$. Then, we can again take $\Inf_{T/R\otimes_kS}$
    with its Hodge filtration which has associated graded pieces
    $\LSym_T^*(\L_{T/R\otimes_kS})$. If this filtration is complete, 
    there is an associated spectral sequence converging to $\pi_*(R\otimes_kS)$.
\end{variant}

\begin{remark}
    If $k\rightarrow R$ is a surjective map of commutative rings with kernel $I$, then the low-degree information
    arising from the Quillen spectral sequence is often the best tool to understand the first few
    homotopy groups of the cotangent complex. Specifically, as
    in~\cite[\href{https://stacks.math.columbia.edu/tag/08RG}{Tag 08RG}]{stacks-project}, there is an exact sequence
    $$\Tor_3^k(R,R)\rightarrow\pi_3(\L_{R/k})\rightarrow\Lambda^2(I/I^2)\rightarrow\Tor_2^k(R,R)\rightarrow\pi_2(\L_{R/k})\rightarrow
    0,$$
    where $\Lambda^2(I/I^2)$ denotes, unusually, the non-derived exterior square of the $R$-module
    $I/I^2$.
    This exact sequence is used in commutative algebra, especially in connection to investigations
    of the cotangent complex. See for example~\cite[Sec.~4]{herzog}.
\end{remark}

\section{Another look at the HKR filtration}\label{sec:hkr}

The next theorem explains how to recover the HKR filtration on $\HH(R/k)$ without the circle
action. From the perspective that $\HH(R/k)$ is the $\Oscr$-cohomology of the loop stack of $\Spec
R$, the theorem says that the adic filtration with respect to constant loops is the HKR filtration.

\begin{theorem}\label{thm:hkr}
    Let $k\rightarrow R$ be a map of derived commutative rings and let $\HH(R/k)\rightarrow R$ be
    the natural augmentation map. Then, there is a natural identification
    $\HH_\fil(R/k)\we\F^\star_\H\Inf_{R/\HH(R/k)}$.
\end{theorem}

\begin{proof}
    Note that $\HH_\fil(R/k)$ is naturally an infinitesimal filtered derived commutative
    $\HH(R/k)$-algebra with $\gr^0$ equivalent to $R$. Thus, there is a natural map
    $\F^\star_\H\Inf_{R/\HH(R/k)}\rightarrow\HH_\fil(R/k)$. It is enough to see that the induced
    map on $\gr^1$ is an equivalence. For this, we can reduce
    to the case of $k=\bZ$ and $R=\bZ[x]$. The cotangent complex
    $\L_{\bZ[x]/\HH(\bZ[x]/\bZ)}$ can be identified with $\bZ[x][2]$ with a generator called
    $\sigma^2\d x$. The formula $\LSym_{\bZ[x]}(\bZ[x][2][-1])$ for the associated graded pieces of
    the Hodge filtration now shows that the spectral sequence (converging to
    $\Infhat_{\bZ[x]/\HH(\bZ[x]/\bZ)}$) collapses at the $\E^1$-page and that the homotopy groups
    of the result are the desired exterior algebra. However, by
    Lemma~\ref{lem:connective_complete}, the filtration is complete.
\end{proof}

\begin{remark}[$S^1$-equivariance]
    Theorem~\ref{thm:hkr} can be upgraded to an $S^1$-equivariant equivalence. Specifically, we see from the
    identification
    $$\F^\star_\H\Inf_{R/\HH(R/k)}\we\F^\star_\H\Inf_{R/R\otimes_{R\otimes_kR}R}\we
    R\otimes_{R\otimes_{\F^\star_\H\Inf_{R/k}}R}R$$
    that $\F^\star_\H\Inf_{R/\HH(R/k)}$ carries a natural circle action. Indeed, it is
    $\HH(R/\F^\star_\H\Inf_{R/k})$, the Hochschild homology of $R$ relative to
    $\F^\star_\H\Inf_{R/k}$ computed internally to $\FMod_k$. This recovers the $S^1$-equivariant HKR
    filtration. It is not $\bT_\fil$-equivariant in the sense that the circle action does not
    induce the de Rham differential on graded pieces. Of course, the filtered circle action also
    exists, but we do not see it directly from this picture.
\end{remark}

\begin{remark}
    Note that the forgetful functor
    $$\DAlg_k\xleftarrow{\F^0}\cMod_{\bT_\fil^\vee}(\FDAlg_k^\inf)$$ factors through the forgetful
    functor
    \begin{equation}\label{eq:forget}
        \cMod_{\bT^\vee}(\FDAlg_k^\inf)\leftarrow\cMod_{\bT_\fil^\vee}(\FDAlg_k^\inf)
    \end{equation}
    obtained by corestriction of scalars along the map $\bT_\fil^\vee\rightarrow\bT^\vee$ of
    infinitesimal filtered derived bicommutative bialgebras, where we use $\ins^0$ to view
    $\bT^\vee$ as such an object.
    It follows that the HKR filtration on $\HH(R/k)$ together with its
    $\bT_\fil^\vee$-coaction is the image of $\ins^0\HH(R/k)$ under the left adjoint
    to~\eqref{eq:forget}.
\end{remark}

\section{Characteristic zero prismatic cohomology}\label{sec:prism}

Nils Wa{\ss}muth wrote a masters thesis~\cite{wasmuth-thesis} under Scholze on prismatic cohomology in characteristic $0$.
In it, several features from $q$-de Rham cohomology appear. We explain this theory entirely by
using the Gauss--Manin connection and also extended it to arbitrary characteristics.
We do not touch upon the more subtle aspects of the theory such as degeneration of the Hodge--de
Rham spectral sequence. For these, see~\cite{andreatta-iovita,bms1,scholze-padic}

\begin{definition}
    Let $k$ be a commutative $\bQ$-algebra and let $A=k\llbracket t\rrbracket$. Set
    $\overline{A}=A/t$ and let $R$ be a smooth commutative $\overline{A}$-algebra. The prismatic
    site $(R/A)_\Prism$ is the opposite of the category pairs $(B,x_B)$ where $B$ is a flat $t$-complete $A$-algebra $B$
    and $x_B\colon\colon R\rightarrow\overline{B}=B/t$ is an $\overline{A}$-algebra homomorphism.
    We equip $(B,x_B)$ with the $t$-completely faithfully flat topology.
    Given an object $(B,x_B)$ we let $\Oscr(B,x_B)=B$ and $\overline{\Oscr}(B,x_B)=B/t$. These
    define sheaves on $(R/A)_\Prism$.
\end{definition}

To prove that this category has finite products and hence is indeed a site takes a little thought;
see~\cite[Sec.~3]{wasmuth-thesis}.
Wa{\ss}muth proves the following theorem on the structure of prismatic cohomology in this context.

\begin{theorem}\label{thm:prism}
    Let $R$ be a smooth commutative $\overline{A}$-algebra.
    \begin{enumerate}
        \item[{\em (i)}] {\bf $t$-de Rham comparison.} The cohomology $\R\Gamma((R/A)_\Prism,\Oscr)$ is computed by the complex
            $$R\llbracket t\rrbracket\xrightarrow{t\d}\Omega^1_{R/k}\llbracket
            t\rrbracket\xrightarrow{t\d}\Omega^2_{R/k}\llbracket t\rrbracket\rightarrow\cdots.$$
        \item[{\em (ii)}] {\bf Hodge--Tate comparison.} There is a natural multiplicative
            equivalence
            $\H^*((R/A)_\Prism,\overline{\Oscr})\iso\Omega^*_{R/k}$.
        \item[{\em (iii)}] {\bf de Rham comparison.} After inverting $t$, one obtains
            $$\R\Gamma((R/A)_\Prism,\Oscr)[t^{-1}]\we\R\Gamma_\dR(R/\overline{A})\Lparen
            t\Rparen.$$
    \end{enumerate}
\end{theorem}

We prove that this form of prismatic cohomology is nothing other than a Hodge-complete derived de
Rham cohomology.

\begin{theorem}\label{thm:our_prism}
    There is an equivalence $\R\Gamma((R/A)_\Prism,\Oscr)\we\dRhat_{R/A}$. The Gauss--Manin
    connection relative to the composition $A\rightarrow R\llbracket t\rrbracket\rightarrow R$
    recovers the $t$-de Rham comparison theorem.
\end{theorem}

\begin{proof}
    Suppose that $(B,x_B)$ is an object of $(R/A)_\Prism$. Then, we can view $B$ with its $t$-adic
    filtration $t^\star B$ as a crystalline---or infinitesimal---filtered derived commutative ring.
    As $\gr^0(t^\star B)\we B/t$ receives a map from $R$, it follows that there is a canonical map
    $\F^\star_\H\dR_{R/A}\rightarrow t^\star B$. Since $B$ is $t$-complete, this map factors
    through a map $\F^\star_\H\dRhat_{R/A}\rightarrow t^\star B$. As this is true functorially in
    $(B,x_B)$, we find in the end a natural map $\F^\star_\H\dRhat_{R/A}\rightarrow
    t^\star\R\Gamma((R/A)_\Prism,\Oscr)$.

    Now, consider the Gauss--Manin connection on $\F^\star_\H\dRhat_{R/A}$ relative to the
    composition $A\rightarrow R\llbracket t\rrbracket\rightarrow R$. We obtain a coherent cochain
    complex of the form
    \begin{equation}\label{eq:tderham}
        \dRhat_{R/R\llbracket t\rrbracket}\rightarrow\dRhat_{R/R\llbracket
        t\rrbracket}\widehat{\otimes}_{R\llbracket t\rrbracket}\Omega^1_{R\llbracket t\rrbracket/A}\rightarrow\cdots.
    \end{equation}
    Since we are in characteristic $0$, we have $\dRhat_{R/R\llbracket t\rrbracket}\we R\llbracket
    t\rrbracket$ and the Hodge filtration is the $t$-adic filtration. The completed tensor products
    allow us to rewrite~\eqref{eq:tderham} as a cochain complex
    \begin{equation}\label{eq:tderham2}
        R\llbracket t \rrbracket\xrightarrow{\nabla}\Omega^1_{R/k}\llbracket
        t\rrbracket\xrightarrow{\nabla}\cdots.
    \end{equation}
    It remains to identify the differential $\nabla$.
    Bringing the Hodge filtration and Griffiths transversality of Theorem~\ref{thm:griffiths} into play we can
    write~\eqref{eq:tderham2} as a filtered cochain complex
    $$t^\star R\llbracket t\rrbracket\xrightarrow{\nabla} t^{\star-1}R\llbracket
    t\rrbracket\otimes_R\Omega^1_{R/k}\xrightarrow{\nabla} t^{\star-2}R\llbracket
    t\rrbracket\otimes_R\Omega^2_{R/k}\xrightarrow{\nabla}\cdots.$$
    This proves that $\nabla$ is divisible by $t$. One can reduce to the case where
    $R=\overline{A}[x]$ and compute that $\nabla(x)=t\d x$.

    It follows that $\dRhat_{R/A}$ also satisfies the $t$-de Rham comparison theorem.
    Therefore, $\H^i(\dRhat_{R/A}\otimes_A\overline{A})\iso\Omega^i_{R/\overline{A}}$ 
    for all integers $i\geq 0$. Since both sides are derived $t$-complete, it follows from the
    Hodge--Tate comparison theorem that
    the natural map $\dRhat_{R/A}\rightarrow\R\Gamma((R/A)_\Prism)_\Oscr$
    is an equivalence.
\end{proof}

\begin{remark}
    The strategy of the proof of Theorem~\ref{thm:our_prism} is not so different from that of
    Wa{\ss}muth who uses a \v{C}ech--Alexander argument in place of the Gauss--Manin connection.
\end{remark}

\begin{remark}
    In mixed characteristic, one can still phrase this definition of prismatic cohomology using
    $t$-complete test objects instead of prisms in the sense of~\cite{prisms}. To do this, it is
    natural to use infinitesimal cohomology. However, the $t$-de
    Rham comparison theorem then will say that $\R\Gamma((R/A)_\Prism)\we\Infhat_{R/A}$ is computed
    by a $t$-de Rham complex of the form $$R\llbracket
    t\rrbracket\rightarrow\Omega^1_{R/k}\llbracket
    t\rrbracket\rightarrow\LSym_R^2(\Omega^1_{R/k}[-1])\llbracket t\rrbracket\rightarrow\cdots.$$
    If de Rham is used instead, then it is not even possible to work over the base $A$ because the
    ideal $(t)$
    does not admit divided powers. One could instead work over a divided power envelope and then
    obtain similar results. This is crystalline cohomology and is very close to the perspective
    of~\cite{bhatt-dejong}. We will return to this point in more detail in future work.
\end{remark}

\small
\bibliographystyle{amsplain}
\bibliography{gm}

\providecommand{\bysame}{\leavevmode\hbox to3em{\hrulefill}\thinspace}
\providecommand{\MR}{\relax\ifhmode\unskip\space\fi MR }
\providecommand{\MRhref}[2]{%
  \href{http://www.ams.org/mathscinet-getitem?mr=#1}{#2}
}
\providecommand{\href}[2]{#2}
\begin{thebibliography}{10}

\bibitem{andreatta-iovita}
Fabrizio Andreatta and Adrian Iovita, \emph{Comparison isomorphisms for smooth
  formal schemes}, J. Inst. Math. Jussieu \textbf{12} (2013), no.~1, 77--151.
  \MR{3001736}

\bibitem{antieau_crystallization}
Benjamin Antieau, \emph{Filtrations and cohomology {I}: crystallization}, arXiv
  preprint arXiv:2511.01567 (2025).

\bibitem{akn-delta}
Benjamin Antieau, Achim Krause, and Thomas Nikolaus, \emph{Prismatic cohomology
  relative to {$\delta$}-rings}, arXiv preprint arXiv:2310.12770 (2023).

\bibitem{ariotta}
Stefano Ariotta, \emph{Coherent cochain complexes and {B}eilinson t-structures,
  with an appendix by {A}chim {K}rause}, arXiv preprint arXiv:2109.01017
  (2021).

\bibitem{barth-peters-vandeven}
W.~Barth, C.~Peters, and A.~Van~de Ven, \emph{Compact complex surfaces},
  Ergebnisse der Mathematik und ihrer Grenzgebiete (3), vol.~4,
  Springer-Verlag, Berlin, 1984. \MR{749574}

\bibitem{bhatt-completions}
Bhargav Bhatt, \emph{Completions and derived de {R}ham cohomology}, arXiv
  preprint arXiv:1207.6193 (2012).

\bibitem{bhatt-dejong}
Bhargav Bhatt and Aise~Johan de~Jong, \emph{Crystalline cohomology and de
  {R}ham cohomology}, arXiv preprint arXiv:1110.5001 (2011).

\bibitem{bms1}
Bhargav Bhatt, Matthew Morrow, and Peter Scholze, \emph{Integral {$p$}-adic
  {H}odge theory}, Publ. Math. Inst. Hautes \'Etudes Sci. \textbf{128} (2018),
  219--397. \MR{3905467}

\bibitem{bhatt-scholze-witt}
Bhargav Bhatt and Peter Scholze, \emph{Projectivity of the {W}itt vector affine
  {G}rassmannian}, Invent. Math. \textbf{209} (2017), no.~2, 329--423.
  \MR{3674218}

\bibitem{prisms}
\bysame, \emph{Prisms and prismatic cohomology}, Ann. of Math. (2) \textbf{196}
  (2022), no.~3, 1135--1275. \MR{4502597}

\bibitem{carlsson}
Gunnar Carlsson, \emph{Derived completions in stable homotopy theory}, J. Pure
  Appl. Algebra \textbf{212} (2008), no.~3, 550--577. \MR{2365333}

\bibitem{connes}
Alain Connes, \emph{Noncommutative differential geometry}, Inst. Hautes
  \'Etudes Sci. Publ. Math. (1985), no.~62, 41--144. \MR{823176}

\bibitem{devalapurkar-bp}
Sanath~K. Devalapurkar, \emph{Lifting to truncated {B}rown-{P}eterson spectra
  and {H}odge--de~{R}ham degeneration in characteristic {$p>0$}}, Forum Math.
  Sigma \textbf{13} (2025), Paper No. e90, 6. \MR{4908648}

\bibitem{dundas-morrow}
Bj\o rn~Ian Dundas and Matthew Morrow, \emph{Finite generation and continuity
  of topological {H}ochschild and cyclic homology}, Ann. Sci. \'Ec. Norm.
  Sup\'er. (4) \textbf{50} (2017), no.~1, 201--238. \MR{3621430}

\bibitem{DwGr}
William~G. Dwyer and John P.~C. Greenlees, \emph{Complete modules and torsion
  modules}, Amer. J. Math. \textbf{124} (2002), no.~1, 199--220. \MR{1879003}

\bibitem{fu_darboux}
Jiaqi Fu, \emph{{$(-1)$}-shifted {D}arboux theorem of derived schemes in
  characteristic {$p>2$}}, arXiv preprint arXiv:2511.12584 (2025).

\bibitem{gaitsgory-rozenblyum-crystals}
Dennis Gaitsgory and Nick Rozenblyum, \emph{Crystals and {D}-modules}, Pure
  Appl. Math. Q. \textbf{10} (2014), no.~1, 57--154. \MR{3264953}

\bibitem{getzler}
Ezra Getzler, \emph{Cartan homotopy formulas and the {G}auss-{M}anin connection
  in cyclic homology}, Quantum deformations of algebras and their
  representations ({R}amat-{G}an, 1991/1992; {R}ehovot, 1991/1992), Israel
  Math. Conf. Proc., vol.~7, Bar-Ilan Univ., Ramat Gan, 1993, pp.~65--78.
  \MR{1261901}

\bibitem{goodwillie_cyclic}
Thomas~G. Goodwillie, \emph{Cyclic homology, derivations, and the free
  loopspace}, Topology \textbf{24} (1985), no.~2, 187--215. \MR{793184}

\bibitem{greenlees-may}
J.~P.~C. Greenlees and J.~P. May, \emph{Derived functors of {$I$}-adic
  completion and local homology}, J. Algebra \textbf{149} (1992), no.~2,
  438--453. \MR{1172439}

\bibitem{griffiths-periods-2}
Phillip~A. Griffiths, \emph{Periods of integrals on algebraic manifolds. {II}.
  {L}ocal study of the period mapping}, Amer. J. Math. \textbf{90} (1968),
  805--865. \MR{233825}

\bibitem{griffiths-discussion}
\bysame, \emph{Periods of integrals on algebraic manifolds: {S}ummary of main
  results and discussion of open problems}, Bull. Amer. Math. Soc. \textbf{76}
  (1970), 228--296. \MR{258824}

\bibitem{grothendieck-derham}
A.~Grothendieck, \emph{On the de {R}ham cohomology of algebraic varieties},
  Inst. Hautes \'Etudes Sci. Publ. Math. (1966), no.~29, 95--103. \MR{199194}

\bibitem{hartshorne-derham}
Robin Hartshorne, \emph{On the {D}e {R}ham cohomology of algebraic varieties},
  Inst. Hautes \'Etudes Sci. Publ. Math. (1975), no.~45, 5--99. \MR{432647}

\bibitem{herzog}
J{\"u}rgen Herzog, Benjamin Briggs, and Srikanth~B Iyengar, \emph{Homological
  properties of the module of differentials}, arXiv preprint arXiv:2502.14159
  (2025).

\bibitem{illusie-cotangent-1}
Luc Illusie, \emph{Complexe cotangent et d\'{e}formations. {I}}, Lecture Notes
  in Mathematics, Vol. 239, Springer-Verlag, Berlin-New York, 1971.
  \MR{0491680}

\bibitem{illusie-derham-witt}
\bysame, \emph{Complexe de de\thinspace {R}ham-{W}itt et cohomologie
  cristalline}, Ann. Sci. \'Ecole Norm. Sup. (4) \textbf{12} (1979), no.~4,
  501--661. \MR{565469}

\bibitem{kaledin-konovalov-magidson}
D.~B. Kaledin, A.~A. Konovalov, and K.~O. Magidson, \emph{Spectral algebras and
  non-commutative {H}odge-to-de {R}ham degeneration}, Tr. Mat. Inst. Steklova
  \textbf{307} (2019), 63--77, English version published in Proc. Steklov Inst.
  Math. {\bf 307} (2019), no. 1, 51--64. \MR{4070058}

\bibitem{kaledin-nhdr}
Dmitry Kaledin, \emph{Non-commutative {H}odge-to-de {R}ham degeneration via the
  method of {D}eligne-{I}llusie}, Pure and Applied Mathematics Quarterly
  \textbf{4} (2008), no.~3, 785--876.

\bibitem{kaledin-spectral}
\bysame, \emph{Spectral sequences for cyclic homology}, Algebra, geometry, and
  physics in the 21st century, Springer, 2017, pp.~99--129.

\bibitem{katz-period}
Nicholas~M. Katz, \emph{On the differential equations satisfied by period
  matrices}, Inst. Hautes \'{E}tudes Sci. Publ. Math. (1968), no.~35, 223--258.
  \MR{242841}

\bibitem{katz-nilpotent}
\bysame, \emph{Nilpotent connections and the monodromy theorem: {A}pplications
  of a result of {T}urrittin}, Inst. Hautes \'{E}tudes Sci. Publ. Math. (1970),
  no.~39, 175--232. \MR{291177}

\bibitem{katz-algebraic}
\bysame, \emph{Algebraic solutions of differential equations ({$p$}-curvature
  and the {H}odge filtration)}, Invent. Math. \textbf{18} (1972), 1--118.
  \MR{337959}

\bibitem{katz-oda}
Nicholas~M. Katz and Tadao Oda, \emph{On the differentiation of de {R}ham
  cohomology classes with respect to parameters}, J. Math. Kyoto Univ.
  \textbf{8} (1968), 199--213. \MR{237510}

\bibitem{krause-nikolaus}
Achim Krause and Thomas Nikolaus, \emph{B\"{o}kstedt periodicity and quotients
  of {DVR}s}, Compos. Math. \textbf{158} (2022), no.~8, 1683--1712.
  \MR{4490929}

\bibitem{liu-wang}
Ruochuan Liu and Guozhen Wang, \emph{Topological cyclic homology of local
  fields}, Invent. Math. \textbf{230} (2022), no.~2, 851--932. \MR{4493328}

\bibitem{htt}
Jacob Lurie, \emph{Higher topos theory}, Annals of Mathematics Studies, vol.
  170, Princeton University Press, Princeton, NJ, 2009. \MR{2522659}

\bibitem{sag}
\bysame, \emph{Spectral algebraic geometry}, available at
  \url{https://www.math.ias.edu/~lurie/papers/SAG-rootfile.pdf}, version dated
  3 February 2018.

\bibitem{mathew-kaledin}
Akhil Mathew, \emph{Kaledin's degeneration theorem and topological {H}ochschild
  homology}, Geom. Topol. \textbf{24} (2020), no.~6, 2675--2708. \MR{4194301}

\bibitem{MNN17}
Akhil Mathew, Niko Naumann, and Justin Noel, \emph{Nilpotence and descent in
  equivariant stable homotopy theory}, Adv. Math. \textbf{305} (2017),
  994--1084.

\bibitem{matsumura}
Hideyuki Matsumura, \emph{Commutative algebra}, second ed., Mathematics Lecture
  Note Series, vol.~56, Benjamin/Cummings Publishing Co., Inc., Reading, MA,
  1980. \MR{575344}

\bibitem{popescu}
Dorin Popescu, \emph{General {N}\'eron desingularization}, Nagoya Math. J.
  \textbf{100} (1985), 97--126. \MR{818160}

\bibitem{quillen-cohomology}
Daniel Quillen, \emph{On the (co-) homology of commutative rings}, Applications
  of {C}ategorical {A}lgebra, Proc. Sympos. Pure Math., vol. XVII, Amer. Math.
  Soc., Providence, RI, 1970, pp.~65--87. \MR{257068}

\bibitem{raksit}
Arpon Raksit, \emph{Hochschild homology and the derived de {R}ham complex
  revisited}, arXiv preprint arXiv:2007.02576 (2020).

\bibitem{scholze-padic}
Peter Scholze, \emph{{$p$}-adic {H}odge theory for rigid-analytic varieties},
  Forum Math. Pi \textbf{1} (2013), e1, 77. \MR{3090230}

\bibitem{simpson_dr}
Carlos Simpson, \emph{Homotopy over the complex numbers and generalized de
  {R}ham cohomology}, Moduli of vector bundles ({S}anda, 1994; {K}yoto, 1994),
  Lecture Notes in Pure and Appl. Math., vol. 179, Dekker, New York, 1996,
  pp.~229--263. \MR{1397992}

\bibitem{stacks-project}
The {Stacks project authors}, \emph{The stacks project},
  \url{https://stacks.math.columbia.edu}, 2025.

\bibitem{swan}
Richard~G. Swan, \emph{N\'eron-{P}opescu desingularization}, Algebra and
  geometry ({T}aipei, 1995), Lect. Algebra Geom., vol.~2, Int. Press,
  Cambridge, MA, 1998, pp.~135--192. \MR{1697953}

\bibitem{tv-book}
Bertrand To{\"e}n and Gabriele Vezzosi, \emph{Derived foliations}, arXiv
  preprint arXiv:2305.08212 (2023).

\bibitem{tsygan-gm}
Boris Tsygan, \emph{On the {G}auss-{M}anin connection in cyclic homology},
  Methods Funct. Anal. Topology \textbf{13} (2007), no.~1, 83--94. \MR{2308582}

\bibitem{wasmuth-thesis}
Nils Wa{\ss}muth, \emph{Prismatic cohomology in characteristic zero}, 2019, MS
  thesis at Bonn.

\end{thebibliography}
\addcontentsline{toc}{section}{References}

\medskip
\noindent
\textsc{Department of Mathematics, Northwestern University}\\
{\ttfamily antieau@northwestern.edu}

\end{document}